\documentclass[a4paper,12pt]{amsart}
\usepackage{cases}
\usepackage{amsmath}
\usepackage{amssymb}
\usepackage[all]{xy}
\usepackage{mathrsfs}
\usepackage{amsfonts}
\usepackage{amscd}
\usepackage{multirow,array}

\usepackage{amssymb}
\usepackage{graphicx, picinpar,epsf}
\usepackage[colorlinks, linkcolor=red,anchorcolor=blue,citecolor=green]{hyperref}

\newcommand{\confer}{{\em cf.}\ }
\renewcommand{\mod}{\operatorname{mod}\nolimits}

\newcommand{\End}{\operatorname{End}\nolimits}

\newcommand{\Ext}{\operatorname{Ext}\nolimits}
\newcommand{\Hom}{\operatorname{Hom}\nolimits}

\newcommand{\Coker}{\operatorname{Coker}\nolimits}
\renewcommand{\dim}{\operatorname{dim}\nolimits}

\newcommand{\add}{\operatorname{add}\nolimits}

\newcommand{\ct}{{\mathcal T}}
\newcommand{\co}{{\mathcal O}}
\newcommand{\crc}{{\mathscr C}}
\newcommand{\crd}{{\mathscr D}}
\newcommand{\crm}{{\mathscr M}}
\newcommand{\coh}{\operatorname{coh}\nolimits}
\newcommand{\vect}{\operatorname{vect}\nolimits}

\newcommand{\can}{\operatorname{can}\nolimits}
\newcommand{\rk}{\operatorname{rk}\nolimits}
\newcommand{\st}{[1]}
\newcommand{\X}{\mathbb{X}}

\newcommand{\vx}{\vec{x}}
\newcommand{\vc}{\vec{c}}
\newcommand{\eps}{\varepsilon}

\newtheorem{theorem}{Theorem}[section]

\newtheorem{corollary}[theorem]{Corollary}

\newtheorem{example}[theorem]{Example}

\newtheorem{lemma}[theorem]{Lemma}

\newtheorem{proposition}[theorem]{Proposition}
\newtheorem{remark}[theorem]{Remark}

\newtheorem*{Riemann-Roch Formula}{Riemann-Roch Formula}

\newtheorem*{List A}{List A}
\numberwithin{equation}{section}
\newtheorem{List}[theorem]{List}

\begin{document}
\input xy
\xyoption{all}

\title[Tilting objects on weighted projective lines]{Tilting objects on tubular weighted projective lines: a cluster tilting approach}

\author[Chen]{Jianmin Chen}
\address{School of Mathematical Sciences, Xiamen University, Xiamen 361005, P.R.China}
\email{chenjianmin@xmu.edu.cn}

\author[Lin]{Yanan Lin}
\address{School of Mathematical Sciences, Xiamen University, Xiamen
361005, P.R.China} \email{ynlin@xmu.edu.cn}

\author[Liu]{Pin Liu}
\address{Department of Mathematics, Southwest Jiaotong University, Chengdu
610031, P.R.China} \email{pinliu@swjtu.edu.cn}

\author[Ruan]{Shiquan Ruan$^\dag$}
\address{School of Mathematical Sciences, Xiamen University, Xiamen 361005, P.R.China}
\thanks{$^\dag$ Corresponding author}
\email{sqruan@xmu.edu.cn}


\begin{abstract}
Using cluster tilting theory, we investigate tilting objects in the stable category of vector
bundles on a weighted projective line of weight type $(2, 2, 2, 2)$. More precisely, a tilting object consisting of rank-two bundles is constructed via cluster tilting mutation. Moreover, the cluster tilting approach also provides a new method to classify the endomorphism algebras of tilting objects in the category of coherent sheaves and the associated bounded derived category.
\end{abstract}

\maketitle

\tableofcontents


\section{Introduction and main results}

Weighted projective lines have been widely studied in representation theory since their introduction by Geigle and Lenzing \cite{[GL]} in 1987. One of the important properties of weighted projective lines is that they have `canonical' tilting bundles, whose endomorphism algebras are the so-called canonical algebras in the sense of Ringel \cite{[Ri]}. A higher dimensional analog of weighted projective line has been introduced by Herschend-Iyama-Minamoto-Oppermann \cite{[HIMO]}, who introduced the so-called Geigle-Lenzing complete intersections. For a fixed positive integer $d$, the category of coherent sheaves on a Geigle-Lenzing projective space, a Noetherian abelian category with global dimension $d$, was proved to be derived equivalent to the category of finite dimensional modules over a finite dimensional algebra called $d$-canonical algebra by showing that any Geigle-Lenzing projective space has a tilting bundle. In doing so, two important subcategories were introduced: one is the category of vector bundles and the other is its full subcategory of direct sums of line bundles. The $d$-tilting bundles on a Geigle-Lenzing projective space are certain objects in $d$-cluster tilting subcategories of vector bundles. In the present paper, we try to use the rich structure of cluster categories (\confer \cite{[BMRRT], [BKL]}) or more general 2-Calabi-Yau triangulated categories (\confer \cite{[KR]})  to provide an explicit description of tilting bundles. For this, we concentrate on the case that the Geigle-Lenzing complete intersection is of Calabi-Yau type and $d=1$ (i.e. a weighted projective line of tubular type).

Now we describe the main results of the paper in more detail. The terminologies mentioned here will be recalled later. Let $\X$ be a weighted projective line of tubular type, i.e. $\X$ has weight type $(2,3,6)$, $(2,4,4)$, $(3,3,3)$ or $(2, 2, 2, 2)$. Let $\coh\X$ be the category of coherent sheaves on $\X$.  The Grothendieck group $K_0(\coh\X)$ was computed by Geigle and Lenzing \cite{[GL]}. There are two important $\mathbb{Z}$-linear forms, \emph{rank} $\rk$ and \emph{degree} $\deg$, on $K_0(\coh\X)$. Kussin-Lenzing-Meltzer \cite{[KLM2]} proved that the full subcategory $\vect\X$ of $\coh\X$ formed by all vector bundles carries a distinguished exact structure such that $\vect\X$ becomes a Frobenius category, with the system $\mathcal{L}$ of all line bundles (rank-one bundles) as the indecomposable projective-injective objects. Hence the attached stable category $\underline{\vect}\X=\vect\X/[\mathcal{L}]$ is a triangulated category. Moreover, there is an triangle equivalence
$
\underline{\mbox{vect}}\mathbb{X}\simeq D^b(\coh\X).
$
By definition of the stable category, indecomposable vector bundles of rank two are destined to play an important role in the structural theory of $\underline{\vect}\X$. It was proved in \cite{[KLM2]} that when $\X$ has the weight triple $(p_1, p_{2}, p_3)$, there is a tilting object of $\underline{\vect}\X$ consisting of rank-two bundles. As a complementarity, we prove the following.
\begin{theorem}[See Theorem \ref{theorem2}]\label{I:thm1}
Let $\X$ be a weighted projective line of weight type $(2,2,2,2)$. Then there is a basic tilting object in $\underline{\vect}\X$ satisfying that each indecomposable direct summand is a rank-two bundle.
\end{theorem}

For any non-zero object $X\in\coh\X$, the \emph{slope} is defined as $\mu X:=\deg X/\rk X$. It was proved in \cite{[KLM2]} that the suspension functor $\st$ of the triangulated category $\underline{\mbox{vect}}\mathbb{X}$
induces a bijection $\alpha: \mathbb{Q}\to\mathbb{Q}$ on slopes, which is monotonically increasing, and satisfies $q<\alpha(q)$ for each $q\in\mathbb{Q}$. The interval category $\crd_q$ is the full subcategory of $\underline{\vect}\X$ defined as the additive closure of all indecomposables $X$ satisfying $q<\mu X\leq \alpha(q)$. It was shown in \cite{[KLM2]} that the interval category $\crd_q$ is abelian and equivalent to $\coh\X$. Being different from other tubular types, the weighted projective line of weight type $(2,2,2,2)$ has the following special feature, which is basically due to Meltzer \cite[Cor. 10.1.1]{[M]}.
\begin{proposition}[See Proposition \ref{proposition4}]\label{slope feature}
Let $\X$ be a weighted projective line of weight type $(2,2,2,2)$ and $T$ a tilting object in $\underline{\vect}\X$. Then there exists some $q\in\mathbb{Q}$ such that the slope of each indecomposable direct summand of $T$ belongs to the interval $[q, \alpha(q)]$.
\end{proposition}

As mentioned already, our proof of Theorem \ref{I:thm1} is based on cluster tilting theory in a 2-Calabi-Yau triangulated category.
The equivalence $\underline{\mbox{vect}}\mathbb{X}\simeq D^b(\coh\X)$ allows us to consider certain quotient category of $\underline{\mbox{vect}}\mathbb{X}$, the \emph{cluster category} associated to $\X$ (\confer \cite{[BKL]}). Buan-Marsh-Reineke-Reiten-Todorov \cite{[BMRRT]} showed that the cluster tilting theory of the cluster category is more regular than the classical tilting theory. The recursive process of mutation of cluster tilting objects is
closely related to the notion of mutation of quivers introduced by Fomin and Zelevinsky in the fundamental paper \cite{[FZ]}. Our key observation is that the mutation class of the quiver of the canonical algebra of type $(2,2,2,2)$ is finite and it consists of only 4 quivers up to isomorphism. This encourages us to give a classification of all the endomorphism algebras of
tilting objects in $D^{b}(\coh\X)$ in terms of quivers with relations, which goes back to Meltzer \cite{[M]}, but with a different approach.
\begin{theorem}\label{I:thm2} Let $\X$ be a weighted projective line with weight type $(2,2,2,2)$ and $\Lambda$ a finite dimensional $k$-algebra. Then $\Lambda$ is an endomorphism algebra of a tilting object in $D^{b}(\coh\X)$ if and only if $\Lambda$ is isomorphic to one of the algebras in List \ref{list}.
\end{theorem}

We also describe the endomorphism algebras of tilting sheaves in $\coh\X$.
\begin{theorem}\label{I:thm3}Let $\X$ be a weighted projective line with weight type $(2,2,2,2)$ and  $\Lambda'$ a finite dimensional $k$-algebra. Then $\Lambda'$ is an endomorphism algebra of a tilting sheaf in $\coh\X$ if and only if it is isomorphic to some $B_{ij}$ in List \ref{list}.
\end{theorem}

The paper is organized as follows: In Section \ref{Sec Pre}, we recall some
basic results on the category of coherent sheaves on a weighted
projective line of weight type $(2,2,2,2)$. In Section \ref{Sec Slope}, we study the slopes of indecomposable direct summands of tilting objects in the stable category of vector bundles and prove Proposition \ref{slope feature}. Section \ref{Sec Rank} is devoted to proving Theorem \ref{I:thm1}, i.e. we construct a tilting object in the stable category of vector bundles whose indecomposable direct summands are all of rank two via cluster tilting mutation. The classification theorems \ref{I:thm2} and \ref{I:thm3} are proved after we describe all the tilting objects corresponding to a given cluster tilting object in Section \ref{Sec Classification}.

\noindent {\bf Convention.}
Let $k$ be an algebraically closed field. A triangulated category $\mathcal{T}$ is always a $k$-linear $\Hom$-finite triangulated category with suspension functor $\st$. For two objects $X$ and $Y$, the morphism space from $X$ to $Y$ in $\mathcal{T}$ is denoted by $\mathcal{T}(X, Y)$. We write $\Ext^1(X, Y)$ for $\mathcal{T}(X, Y\st)$. For an object $X$, denote by $|X|$ the number of non-isomorphic indecomposable direct summands of $X$. Denote by $\add X$ the subcategory of $\ct$ consisting of objects which are direct summands of finite direct sums of $X$.
Throughout the paper, we view isomorphism as equality for notational simplicity.

\section{Preliminaries}\label{Sec Pre}
\subsection{Quiver mutation and cluster tilting mutation}
Let us recall that a quiver is an oriented graph. A {\it loop} in a quiver $Q$ is an arrow whose source coincides with its target. Let $Q$ be a finite quiver without loops or oriented cycles of length 2 (2-cycles for short). Let $i$ be a vertex of $Q$. The mutation of the quiver $Q$ at the vertex $i$ is a quiver denoted by $\crm_i(Q)$ and obtained from $Q$ as follows:
\begin{itemize}
\item[(M1)] for any couple of arrows $j \to i \to k$, add an arrow $j \to k$;
\item[(M2)] reverse the arrows incident with $i$;
\item[(M3)] remove a maximal collection of 2-cycles.
\end{itemize}

\begin{example}\label{example}
Let $Q$ be the following quiver:
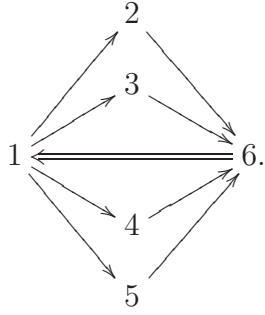
\[
\begin{array}{cccc}
 \xymatrix@-1pc
 {&& 2\ar[rrdd] \\
  &&3\ar[rrd]\\
    1\ar[rruu] \ar[rru]\ar[rrd] \ar[rrdd]  & &&&6.\ar@2{->}[llll]\\
  && 4\ar[rru] \\
  && 5\ar[rruu]
  }
  \end{array}
\]It is easy to compute that the mutation class (the set of all quivers obtained from $Q$ by iterated mutations)
consists of the following 4 quivers up to isomorphism.

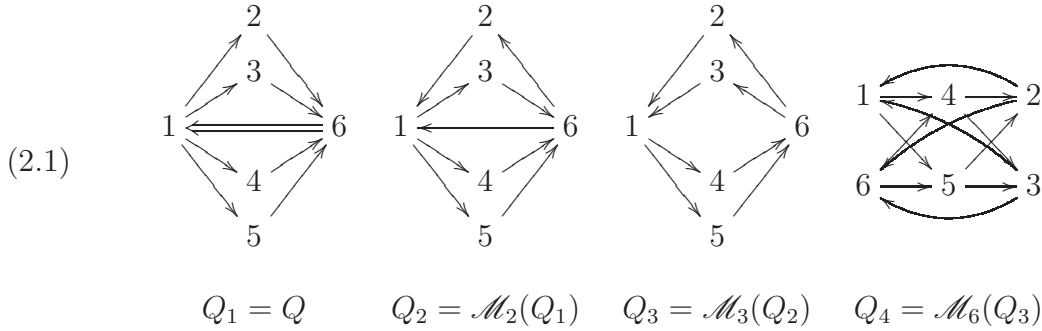
\begin{equation}\label{quiver mutation classes}\begin{array}{cccc}
 \xymatrix@-1.5pc
 {&&2 \ar[rrdd] \\
  &&3 \ar[rrd]\\
  1  \ar[rruu] \ar[rru]\ar[rrd] \ar[rrdd]  & &&&6\ar@2{->}[llll]\\
  &&4\ar[rru] \\
  &&5 \ar[rruu]
  }
& \xymatrix@-1.5pc
{&&2\ar[lldd] \\
  &&3 \ar[rrd] \\
  1\ar[rru] \ar[rrd] \ar[rrdd]  &&&&6 \ar[lluu] \ar[llll]\\
  &&4 \ar[rru] \\
  &&5 \ar[rruu]
  }
& \xymatrix@-1.5pc
 {&&2\ar[lldd]\\
  &&3 \ar[lld] \\
  1 \ar[rrd] \ar[rrdd]   & &&&6\ar[lluu] \ar[llu]\\
  &&4 \ar[rru]\\
  &&5 \ar[rruu]
  }
  & \xymatrix@-1.5pc
{\\
\\
1\ar[rr]^{}\ar[rrdd]^>>>>>{}&&4\ar[rr]^{}\ar[rrdd]^>>>>>{}&&2\ar@/_1pc/[llll]\ar@/_/[ddllll] \\
  \\
6\ar[rr]^{}\ar[rruu]_>>>>>{}&&5\ar[rr]^{}\ar[rruu]_>>>>>{}&&3\ar@/^1pc/[llll]\ar@/_/[uullll]
}\\[8em]
Q_{1}=Q & Q_{2}=\crm_2(Q_1) &Q_{3}=\crm_3(Q_2) & Q_{4}=\crm_6(Q_3)
\end{array}
\end{equation}
\end{example}

Recall that for a triangulated category $\mathcal{T}$, a Serre functor is an autoequivalence $\Sigma: \mathcal{T}\to \mathcal{T}$ such that there is a bifunctorial isomorphism
\[
\ct(X, Y)=D\ct(Y, \Sigma X), \ \forall X, Y\in\mathcal{T},
\]where $D=\Hom_k(-, k)$ is the usual duality over the ground field $k$. A triangulated category $\mathcal{T}$ is \emph{$2$-Calabi-Yau} if it admits a Serre functor isomorphic as a triangle functor to the $2$nd power of its suspension functor. It follows that we have a bifunctorial isomorphism
\[
\Ext^1(X, Y)=D\Ext^1(Y, X), \ \forall X, Y\in\mathcal{T}.
\]An object $T$ in a 2-Calabi-Yau triangulated category $\ct$ is \emph{cluster tilting} if
\begin{itemize}
\item[(a)] $T$ is rigid, i.e. $\Ext^1(T, T)=0$ and
\item[(b)] for each object $X$ of $\ct$, $\Ext^1(T, X)=0$ implies that $X\in\add T$.
\end{itemize}

It is well-known that from a cluster tilting object in a 2-Calabi-Yau triangulated category $\ct$, it is possible to construct others by a recursive process resumed in the following.
\begin{theorem}[\cite{[BMRRT],[IY]}]\label{cluster mutation}Let $\ct$ be a 2-Calabi-Yau triangulated category with a cluster tilting object $T$. Let $T_i$ be indecomposable and
$T=\bar{T}\oplus T_i$. Then there exists a unique indecomposable $T_i^*$
non-isomorphic to $T_i$ such that $\crm_{T_i}(T)$=$\bar{T}\oplus T_i^*$ is cluster
tilting. Moreover $T_i$ and $T^*_i$ are linked by the existence of
exchange triangles
\[T_i\xrightarrow{u}B\xrightarrow{v}T^*_i\xrightarrow{w}T_i\st\quad\text{and}\quad
T^*_i\xrightarrow{u'}B'\xrightarrow{v'}T_i\xrightarrow{w'}T^*_i\st,
\]
where $u$ and $u'$ are minimal left $\add \bar{T}$-approximations and $v$ and $v'$ are minimal
right $\add \bar{T}$-approximations.
\end{theorem}

The recursive process of mutation of cluster tilting objects is
closely related to the notion of mutation of quivers in the
following sense.
\begin{theorem}[\cite{[BIRS]}]\label{theorem BIRS}
Let $\ct$ be a 2-Calabi-Yau triangulated category with a
cluster tilting object $T$. Let $T_i$ be an indecomposable direct
summand of $T$, and denote by $T'$ the cluster tilting object
$\crm_{T_i}(T)$. Denote by $Q_T$ (resp. $Q_{T'}$) the quiver of the
endomorphism algebra $\End_\ct(T)$ (resp. $\End_\ct(T')$). Assume
that there are no loops and no 2-cycles at the vertex $i$ of $Q_T$
(resp. $Q_{T'}$) corresponding to the indecomposable $T_i$ (resp.
$T^*_i$). Then we have $Q_{T'}=\crm_i(Q_T).$
\end{theorem}

\subsection{Coherent sheaves on a weighted projective line of weight type $(2,2,2,2)$}
We follow \cite{[GL]}. Let $\lambda$ be a point in the projective line $\mathbb{P}^1$ over the ground field $k$. We assume that $\lambda\not=0, 1, \infty$.
Let $\mathbb{L}$ be the rank one abelian group with
generators $\vec{x}_{1}, \vec{x}_{2}, \vx_3, \vec{x}_{4}$ and defining
relations
\[ 2\vx_1=2\vx_2=2\vx_3=2\vx_4=:\vec{c}.\]The element $\vec{c}$ is called the \emph{canonical element} of
$\mathbb{L}$, and each element $\vec{x}\in \mathbb{L}$ can be
uniquely written in normal form
\[\vec{x}=\sum\limits_{i=1}^{4}l_{i}\vec{x}_{i}+l\vec{c},
\ \ \mbox{where} \ \ 0\leq  l_{i}< 2 \ \ \mbox{and} \ \ l\in
\mathbb{Z}.\] For any $\vec{x}\in \mathbb{L}$, define $\vec{x}\geq
0$ if $l\geq 0$ in the normal form of $\vec{x}$. Then $\mathbb{L}$
becomes a partial order group, and each $\vec{x}\in \mathbb{L}$
satisfies exactly one of the two possibilities: \[ \vec{x}\geq 0 \ \
\text{or}\ \ \vec{x}\leq \vec{\omega}+\vec{c},\] here
$\vec{\omega}=2\vec{c}-\sum\limits_{i=1}^{4}\vec{x}_{i}$ is
called the \emph{dualizing element} of $\mathbb{L}$.
There is a group homomorphism $\delta:\mathbb{L}\to\mathbb{Z}$ determined by $\delta(\vx_i)=1$ for $1\leq i\leq 4$.

Denote by $S$ the commutative algebra
\[S=k[X_{1}, X_2, X_3, X_4]/I:= k[x_{1},x_{2}, x_3, x_{4}],\] where
$I$ is the ideal generated by
$f_{1}=X_{3}^{2}-X_{2}^{2}+X_1^2$ and $f_2=X_4^2-X_2^2+\lambda X_1^2$. Then $S$ is $\mathbb{L}$-graded by setting
$\mbox{deg}(x_{i})=\vx_i \ (i=1,2,3,4)$, and $S$ carries a
decomposition into $k$-subspaces:
\[S=\bigoplus\limits_{\vec{x}\in
\mathbb{L}}S_{\vec{x}}.\]

The category of coherent sheaves on $\X$ can be defined as the
quotient of the category of finitely generated $\mathbb{L}$-graded
$S$-modules over the Serre subcategory of finite length modules as
follows
\[\coh\X:=\mod^{\mathbb{L}}(S)/\mbox{mod}_{0}^{\mathbb{L}}(S).\]
The free module $S$ gives the structure sheaf $\co$, and each line
bundle is given by the grading shift $\co(\vec{x})$ for a uniquely
determined element $\vec{x}\in \mathbb{L}$. Moreover, there is a
natural isomorphism
\[\Hom(\co(\vec{x}), \co(\vec{y}))=S_{\vec{y}-\vec{x}}.
\]

Denote by $\vect\X$ the full subcategory of $\coh\X$ formed by all
vector bundles, i.e. locally free sheaves, and by $\coh_0\X$ the
full subcategory formed by all sheaves of finite length, i.e.
torsion sheaves. Geigle and Lenzing \cite{[GL]} showed that each
coherent sheaf decomposes as a direct sum of a vector bundle and a torsion sheaf, and there are no non-zero morphisms from $\coh_0\X$ to
$\mbox{vect}\mathbb{X}$. Moreover, $\coh\X$ is a hereditary abelian
category with Serre duality of the form
\[D\Ext^1(X, Y)=\Hom(Y, X(\vec{\omega})),
\] which implies the existence of almost split sequences for $\coh\X$
with the Auslander-Reiten translation $\tau$ given by the grading
shift with $\vec{\omega}$.

The Grothendieck group $K_0(\X)$ of $\coh\X$ was computed by Geigle
and Lenzing \cite{[GL]}, and it was proved to be the vector space
with basis indexed by elements $\co(\vec{x})$ with $0\leq\vec{x}\leq
\vec{c}$, where we still write $X\in K_0(\X)$ for the class of an
object $X\in \coh\X$.
The Euler form on $K_0(\X)$ is defined as follows on classes of objects $X,Y\in\coh\X$:
\[\langle X,Y\rangle=\dim_{k}\Hom(X,Y)-\dim_{k}\Ext^{1}(X,Y).
\]

There are some important $\mathbb{Z}$-linear maps on $K_{0}(\mathbb{X})$, including \emph{determinant} $\det$, \emph{rank}
$\rk$ and \emph{degree} $\deg$. The determinant map is the group homomorphism $\det:
K_0(\X)\to \mathbb{L}$ given by
$\det(\co(\vec{x}))= \vec{x}.$ The degree
function is the composition of $\delta$ and $\det$, that is, it is
determined by
$\mbox{deg}(\co(\vec{x}))=\delta(\vec{x}).$
The rank function $\mbox{rk}: K_{0}(\mathbb{X})\rightarrow
\mathbb{Z}$ is characterized by
$\mbox{rk}(\co(\vec{x}))=1.$
For each non-zero object $X\in \coh\X$, define the \emph{slope} of
$X$ as $\mu X=\frac{\mbox{deg} X}{\mbox{rk}X}.$

Notice that the rank is strictly positive for a non-zero vector bundle and
vanishes for a torsion sheaf. The slope of a
vector bundle belongs to $\mathbb{Q}$, while it is infinity for a
torsion sheaf.
By \cite{[GL]}, for any two indecomposable objects $X, Y$ in
$\coh\X$,
\begin{equation}\label{no morphism from right to left}
\mbox{Hom}(X,Y)\neq 0
 \text{\ implies\ } \mu X\leq \mu Y.
\end{equation}
Moreover, we have the following result.
\begin{theorem}[Riemann-Roch Formula, \cite{[LM]}]
For each $X, Y\in \coh\mathbb{X}$, we have
\begin{equation}\label{RR-formula}
 \langle
X\oplus \tau X,Y\rangle=\rk X\deg Y-\deg X\rk Y=\rk X\rk Y(\mu Y-\mu X).
\end{equation}
\end{theorem}

\subsection{Stable category of vector bundles and the associated cluster category}\label{section frobenius category}

Recall from \cite{[KLM2]} that a sequence
\[0\rightarrow X^{\prime}\rightarrow X\rightarrow
X^{\prime\prime}\rightarrow 0\] in $\mbox{vect}\mathbb{X}$ is called
\emph{distinguished exact} if for each line bundle $L$ the induced
sequence
\[
0\rightarrow \Hom(L,X^{\prime})\rightarrow \Hom(L, X)\rightarrow
\Hom(L, X^{\prime\prime})\rightarrow 0\] is exact. Kussin, Lenzing
and Meltzer \cite{[KLM2]} proved that the distinguished exact
sequences define a Frobenius exact structure on the category
$\vect\X$, such that the system of all line bundles is the system of
all indecomposable projective-injectives. By a general result of
\cite{[H]}, the related stable category
\[\underline{\vect}\X=\vect\X/[\mathcal{L}]
\] is a triangulated category. The suspension functor $\st$ is given by the formation of co-syzygies.
It was proved in \cite{[KLM2]} that there is a triangle equivalence
\begin{equation}\label{II:equ}
\underline{\mbox{vect}}\mathbb{X}\simeq D^b(\coh\X).
\end{equation}
For simplification of notations, in the rest of the paper we denote
the stable category $\underline{\vect}\X$ by $\mathscr{D}$.

\begin{theorem}[\cite{[KLM2]}]\label{Serre duality}
\begin{itemize}
\item[(1)]$\mathscr{D}$ is
$\Hom$-finite, Krull-Schmidt and homologically finite.
\item[(2)]$\crd$ has Serre duality: For any two objects $X$ and $Y$ in $\crd$,
\[D\Ext^1(X, Y)=\mathscr{D}(Y, X(\vec{\omega})),
\]In particular, $\crd$ has Auslander-Reiten triangles, and the grading shift by $\vec\omega$ also serves
as the Auslander-Reiten translation $\tau$ for $\crd$.
\end{itemize}
\end{theorem}

As in \cite{[BKL]}, the cluster category $\crc$ associated to $\X$ is defined to be the orbit category $\crd/G^{\mathbb{Z}}$ under the action of the cyclic group generated by the  auto-equivalence $G=\tau^{-1}[1]$, where $\tau$ denotes the Auslander-Reiten translation
and $\st$ denotes the suspension functor of $\crd$.
More precisely, the cluster category $\crc$ has the same
objects as $\crd$, and for any objects $X, Y$, morphism spaces are
given by \[\crc(X, Y)=\bigoplus_{n\in \mathbb{Z}}\mathscr{D}(X,
G^{n}Y)\] with the obvious composition. This orbit category is a
2-Calabi-Yau triangulated category and the canonical
functor $\pi: \crd\to\crc$ is a triangle functor. We still denote by $\st$ the suspension functor of $\crc$.

Recall that a cluster tilting object $T'$ is \emph{reachable} from $T$ if there is a sequence of mutations
\[
T=T^{(0)}\rightsquigarrow T^{(1)}\rightsquigarrow \cdots\rightsquigarrow T^{(N)}=T'
\]such that the quiver of the endomorphism algebra $\End(T^{(i)})$ has neither loops nor 2-cycles for any $1\leq i\leq N$. Theorem \ref{theorem BIRS} implies that if a cluster tilting object $T'$ is reachable from $T$, then the quiver of the endomorphism algebra of $T'$ is mutation-equivalent to the quiver $Q$ of the endomorphism algebra of $T$. In particular, in the cluster category $\crc$ concerned, all quivers mutation-equivalent to $Q$ are obtained in this way (\confer
\cite{[BKL]}).

\section{The slope features}\label{Sec Slope}
Let $\X$ be a weighted projective line of weight type $(2,2,2,2)$. In this section, we present key features of indecomposable direct summands of tilting objects in the stable category $\crd=\underline{\vect}\X$ with respect to the slope.
\subsection{The interval category}
It was proved that the suspension functor $\st$ of $\crd$ induces a bijection $\alpha:\mathbb{Q}\to \mathbb{Q}$ on slopes, which is monotonically increasing, and satisfies $q<\alpha(q)$ for each $q\in\mathbb{Q}$. The following tubular factorization property is useful.
\begin{lemma}[{\cite[Thm. A.4]{[KLM2]}}]\label{lemma2} Let $X$ and $Y$ be indecomposable
in $\vect\X$ with slopes $\mu(X)=q$ and $\mu(Y)=q'$. If
$q'>\alpha(q)$ then every morphism $X\to Y$ factors through a direct
sum of line bundles.
\end{lemma}

For any $q\in\mathbb{Q}$, the \emph{interval category} $\crd_q$ is the full subcategory of $\crd$
obtained as the additive closure of all the indecomposable objects
with slopes in the half-open interval $(q, \alpha(q)]$. It was proved in \cite{[KLM2]} that $\crd_q$ is an abelian category
and there is an equivalence
\[\Phi_q:\crd_q\xrightarrow{\sim}\coh\X.\]
We first show that $\Phi_{q}$ preserves
the order of slopes.

\begin{lemma}\label{lemma6} For any indecomposable objects $F_{1},F_{2}\in \crd_q$,
\[
\mu F_{1}\leq\mu F_{2}\ \text{if and only if}\ \mu(\Phi_{q}(F_{1}))\leq\mu
(\Phi_{q} (F_{2})). \] Consequently,
\[
\mu F_{1}=\mu F_{2}\ \text{if and only if}\ \mu(\Phi_{q}(F_{1}))=\mu
(\Phi_{q} (F_{2})).
\]
\end{lemma}

\begin{proof} We first show that $\Phi_{q}$ commutes with the Auslander-Reiten translations. Recall that each connected component of the
Auslander-Reiten quiver of $\crd_q$ is a homogeneous tube of rank one or two.
Hence we only need to show that $\Phi_{q}(\tau E)=\tau \Phi_{q}(E)$ for any quasi-simple object $E\in \crd_q$.
Obviously, as an equivalence $\Phi_{q}$ preserves the quasi-simple objects, hence $\Phi_{q}(E)$ is quasi-simple. Observe that $$\Ext^1(E,  \tau E)\neq 0\neq \Ext^1(\tau E,  E),$$ it follows that \[\Ext^1(\Phi_{q}(E), \Phi_{q}( \tau E))\neq 0\neq \Ext^1(\Phi_{q}( \tau E),  \Phi_{q}(E)),
\] hence $\Phi_{q}(E)$ and $\Phi_{q}(\tau E)$ belong to the same tube.
Now using the following equivalences $$E\not\cong \tau E\Leftrightarrow\crd(E,  \tau E)= 0\Leftrightarrow\Hom(\Phi_{q}(E),\Phi_{q}( \tau E))=0\Leftrightarrow\Phi_{q}( E)\not\cong\Phi_{q}(\tau  E),$$ we conclude that $\Phi_{q}(\tau E)=\tau \Phi_{q}(E)$.
Consequently, $\Phi_{q}$ preserves rank-one tubes and rank-two tubes respectively.

Now we prove $\mu F_{1}\leq\mu F_{2}$ if and only if $\mu(\Phi_{q}(F_{1}))\leq\mu
(\Phi_{q} (F_{2}))$.
Observe that if $F_{1}, F_{2}$ belong to the same tube, then so do $\Phi_{q}(F_{1})$ and $\Phi_{q}(F_{2})$, hence $\mu(\Phi_{q}(F_{1}))=\mu
(\Phi_{q} (F_{2}))$. Without loss of generality, we assume $F_{1}, F_{2}$ belong to different tubes in the following. Then by Riemann-Roch Formula (\ref{RR-formula}), we have
\[
\begin{aligned}
 \mu F_{1}\leq\mu F_{2}&
\Longleftrightarrow\crd(F_{2},
F_{1}\oplus \tau F_{1})=0 \\
&\Longleftrightarrow\Hom(\Phi_{q}(F_{2}), \Phi_{q}(F_{1})\oplus \tau (\Phi_{q}(F_{1})))= 0
\\
&\Longleftrightarrow\mu(\Phi_{q}(F_{1}))\leq\mu
(\Phi_{q} (F_{2})).
\end{aligned}
\]
Consequently, we have $\mu F_{1}=\mu F_{2}$ if and only if $\mu(\Phi_{q}(F_{1}))=\mu
(\Phi_{q} (F_{2}))$.
\end{proof}

The following is an easy consequence.

\begin{corollary}\label{proposition3}
Let $F_{1},F_{2}$ be indecomposable objects in $\crd_q$. If $\mu F_{1}<\mu F_{2}$, then
$\crd(F_{1}, F_{2}\oplus\tau F_{2})\neq 0.$
\end{corollary}

\subsection{Slopes of indecomposable direct summands of tilting objects}
Recall that a sheaf $T$ in $\coh\X$ is called tilting, if
\begin{itemize}
\item[-] $T$ is rigid, i.e. $\Ext^1(T, T)=0$ and
\item[-] for any $X\in\coh\X$, the condition $\Ext^1(T, X)=0=\Hom(T, X)$ implies that $X=0$.
\end{itemize}
Similarly, in the stable category $\crd$, we say that an object $T$ is \emph{extension-free} if $\mathscr{D}(T, T[n])=0$ for each non-zero integer $n$.
An extension-free object $T\in\mathscr{D}$ is \emph{tilting} if for each non-zero object $X$, there exists some integer $n$ such that $\mathscr{D}(T, X[n])\neq
0.$
It was proved in \cite{[CLR]} that each basic tilting object $T$ in $\crd$ contains 6 indecomposable direct summands, i.e. $|T|=6$.

By using the equivalence $\Phi_q:\crd_q\xrightarrow{\sim}\coh\X$, we have the following results.

\begin{lemma}\label{lemma3}Let $X$ and $Y$ be indecomposable objects in $\crd_q$.
Then $\crd(X, Y[n])=0$ for any integer $n\not= 0, 1$.
In particular,  $T\in\crd_q$ is extension-free in $\mathscr{D}$ if and only
if $\mathscr{D}(T, T[1])=0$.
\end{lemma}

\begin{lemma}\label{Phi_q preserves tilting}
Let $T$ be an object in $\crd_q$. Then $T$ is tilting in $\crd$ if and only if $\Phi_q(T)$ is a tilting sheaf in $\coh\X$.
\end{lemma}

As mentioned before, being different from other tubular types, the direct summands of a tilting object for weight type $(2,2,2,2)$ have the following slope feature, which is basically due to Meltzer \cite[Cor. 10.1.1]{[M]}.

\begin{proposition}\label{proposition4}
Let $\X$ be a weighted projective line with weight type $(2,2,2,2)$ and $T=\bigoplus\limits_{i=1}^{6}T_{i}$ a basic tilting object in $\underline{\vect}\X$. For any $i=1,\cdots, 6$,  the slope $\mu T_i$ belongs to the closed interval $[q, \alpha(q)]$ for some $q\in\mathbb{Q}$.

\end{proposition}

\begin{proof}
Assume that ${\rm max}\{\mu T_i \,|\, 1\leq i\leq 6\}=\alpha(q)$ for some $q$. We need to show that $\mu T_i\geq q$ for each $i$. Under the equivalence $\Phi_q:\crd_q\xrightarrow{\sim}\coh\X$, we can identify the stable category $\crd$ with $D^b(\coh\X)$, where the extended-closed subcategory of $\crd$ generated by the indecomposable objects of slope $\alpha(q)$ (resp. $q$) corresponds to the torsion subcategory $\coh_0\X$ (resp. $\coh_0\X[-1]$). Then by Corollary 10.1.1 in \cite{[M]}, each summand $T_i$ corresponds to a coherent sheaf or a stalk complex $V[-1]$ for some torsion sheaf $V$. It follows that $\mu T_i\geq q$, we are done.
\end{proof}

\section{Tilting object with rank-two indecomposable direct summands}\label{Sec Rank}
Let $\X$ be a weighted projective line with weight type $(2,2,2,2)$. This section provides an explicit tilting object in the stable category consisting of only rank-two bundles. We still denote by $\crd$ the stable category $\underline{\vect}\X$ of vector bundles and by $\crc$ the associated cluster category. By the definition of the cluster category, we use the same notation for an object in $\crd$ and its image in $\crc$ under the canonical functor $\pi: \crd\to\crc$.

Recall that $\coh\X$ is a hereditary abelian
category with Serre duality of the form
\[D\Ext^1(X, Y)=\Hom(Y, X(\vec{\omega})),
\] which implies the existence of almost split sequences for $\coh\X$
with the Auslander-Reiten translation $\tau$ given by the grading
shift with $\vec{\omega}$.

\subsection{An initial cluster tilting object} Let $E$ be the Auslander bundle determined by the almost split sequence
\[0\to \co(\vec{\omega}) \to
E\to \co \to 0.\] For each $i=1,\cdots,4$, let $E_{i}$ be the central term of the following non-split exact
sequence
\[0\to \co(\vec{\omega}) \to
E_{i} \to \co(\vec{x}_{i})
\to 0. \]
We remind that such an exact sequence is unique up to isomorphism, and $E_{i}$ is denoted by $E\langle\vec{x}_i\rangle$ in \cite{[KLM2]}. Set
\[F=E(\vec{w}+\vec{c})[-1].
\]
Then $F$ has rank 3, which fits into the following exact sequence for each $1\leq i\leq 4$ (\confer\cite[Sec. 6]{[CLR]}):
\begin{equation}\label{exact sequences lemma}
0 \to
\co(\vec{\omega})\to F
\to E(\vec{\omega} +
\vec{x}_{i})\to 0.
\end{equation}
Moreover, by \cite[Thm. 6.2]{[CLR]}, the object
\[T_{\can}=E\oplus(\bigoplus\limits_{ i=1}^{4}E_{i} )\oplus
F\] is a tilting object in $\crd$ and the endomorphism algebra
$\End_{\crd}(T_{\can})$ is a canonical algebra of type (2,2,2,2).

\begin{proposition}\label{canonical cluster tilting}
The image of the tilting object $T_{\can}$ under the canonical functor $\pi:\crd\to\crc$ is a cluster tilting object in $\crc$.
\end{proposition}

\begin{proof} Let $X, Y$ be two indecomposable direct summands of $T_{\can}$. Note that the slopes $\mu(X), \mu(Y)$ belong to the interval $(\alpha^{-1}(1), 1]$. By Lemma \ref{Phi_q preserves tilting}, $\Phi_{\alpha^{-1}(1)}(T_{\can})$ is a tilting sheaf in $\coh\X$. According to \cite[Proposition 2.3]{[BKL]}, tilting sheaves in $\coh\X$ coincide with cluster tilting objects in the cluster category. Hence $\pi(T_{\can})$ is a cluster tilting object.
\end{proof}

\subsection{Exchange triangles}
To apply cluster tilting mutation, we shall make frequent use of the following results.
\begin{lemma}\label{exchange triangle1}For each $i=1,\cdots, 4$, there is a triangle in $\crc$
\[E_i\to F\to E(\vec{\omega}+\vec{x}_i)\to E_i\st.
\]
\end{lemma}
\begin{proof}For any $i$, there is an almost split sequence in $\coh\X$
\begin{equation}\label{exact sequence of 4.2}
0\to\co(\vec{x}_{i})\xrightarrow{\iota_i}
E(\vec{\omega}+\vec{x}_{i})\to\co(\vec{\omega}+\vec{x}_{i})\to 0.
\end{equation}
Applying the functor $\Hom(-, \co(\vec{\omega}))$ to (\ref{exact
sequence of 4.2}), we get
\[
\Ext^{1}(E(\vec{\omega}+\vec{x}_{i}),
\co(\vec{\omega}))=\Ext^{1}(\co(\vec{x}_{i}), \co(\vec{\omega}))=k.
\]Then by (\ref{exact sequences lemma}), there exists a commutative diagram induced by pullback of $\iota_i$ and $\phi$:
\[
\xymatrix{
0 \ar[r] & \co(\vec{\omega}) \ar@{=}[d]\ar[r] & E_{i}\ar[d]^-{a_{i}} \ar[r]\ar@{}[dr]|-{\circlearrowleft} &  \co( \vec{x}_{i}) \ar[d]^-{\iota_i}\ar[r] & 0  \\
0 \ar[r] & \co(\vec{\omega}) \ar[r] & F\ar[r]^-{\phi}  & E(\vec{\omega}+\vec{x}_{i}) \ar[r]& 0 \ . \\
 }
\]
Moreover, we know that the right square is also a pushout. Hence
$a_{i}: E_{i}\to F$ is injective and
$$\Coker(a_{i})= \Coker(\iota_i) =
\co(\vec{\omega}+\vec{x}_{i}).$$ Then
we obtain the following exact sequence:
\begin{equation}\label{exact sequence of 4.3}
0\to E_{i}\xrightarrow{a_{i}} F\to\co(\vec{\omega}+\vec{x}_{i})\to
0.
\end{equation}
Denote by $I(E_{i})$ the injective hull of $E_{i}$. By \cite{[CLR]},
\[
I(E_{i})=\co(\vec{x}_{i})\oplus(\bigoplus_{j\neq i}\co(\vec{\omega}+\vec{x}_{j})).
\]
Consider the following pushout commutative diagram
\[
\xymatrix{
  0 \ar[r] & E_{i}\ar@{}[dr]|-{\circlearrowright}    \ar[d] \ar[r] & F \ar[d] \ar[r] & \co(\vec{\omega}+\vec{x}_{i}) \ar@{=}[d] \ar[r] & 0 \\
  0 \ar[r] & I(E_{i}) \ar[r]  & C \ar[r] & \co(\vec{\omega}+\vec{x}_{i}) \ar[r] & 0.}
\]
Notice that  for any $j\neq i$,
\[
\Ext^{1}(\co(\vec{\omega}+\vec{x}_{i}), \co(\vec{\omega}+\vec{x}_{j}))=0.
\]
Combining with (\ref{exact sequence of 4.2}), we get
\[
C=E(\vec{\omega}+\vec{x_{i}})\oplus(\bigoplus\limits_{j\neq i}\co(\vec{\omega}+\vec{x}_{j})) .
\]Hence, there exists a triangle in $\crd$
\[
  E_{i}\to F\to E(\vec{\omega}+\vec{x}_{i}) \to E_{i}[1].
\]
Since $\pi: \crd\to\crc$ is a triangle functor, we get what we want.
\end{proof}

\begin{lemma}\label{exchange triangle2}The following is a triangle in $\crc$
\[
F[-1]\to E(\vec{x}_1-\vec{x}_2)\to E_3\oplus E_4\to F.
\]
\end{lemma}
\begin{proof}From (\ref{exact sequence of 4.3}), we obtain the following commutative diagram induced by the pullback of $a_{3}$
and $a_{4}$
\[
\xymatrix{
0 \ar[r] & A \ar@{}[dr]|-{\circlearrowleft}\ar[d]_-{b_{3}}\ar[r]^-{b_{4}} & E_{4}\ar[d]_-{a_{4}} \ar[r] & B \ar[d]_-{b}\ar[r] & 0  \\
0 \ar[r] & E_{3}\ar[r]^-{a_{3}} & F\ar[r]  & \co(\vec{\omega}+\vec{x}_{3})  \ar[r]& 0,\\
   }
\]where the induced maps $b_{3}, b_{4}$ and $b$ are all injective.
By the Snake Lemma we have an surjection $\Coker(a_4)=\co(\vec{\omega}+\vec{x}_{4})\twoheadrightarrow \Coker(b)$.
Observe that there are no morphisms between $\co(\vec{\omega}+\vec{x}_{3})$ and $\co(\vec{\omega}+\vec{x}_{4})$. We obtain that $\rk(\Coker(b))=0$ and then $\rk B=1$.
Notice that $\mu E_{4}=\frac{1}{2}$ and $\mu (\co(\vec{\omega}+\vec{x}_{3}))=1$.
We get $B=\co(\vec{\omega}+\vec{x}_{3})$.
It follows that the left square is also a
pushout, and
$\det A=\vec{x}_3-\vec{x}_4$ and $\rk A=1$, which ensures that
$A=\co(\vec{x}_{3}-\vec{x}_{4})$. Thus there is an exact sequence in $\coh\X$
\[
0\to\co(\vec{x}_{3}-\vec{x}_{4})\to E_{3}\oplus E_{4}\to F\to 0.
\]Denote by $P(F)$ the projective cover of $F$. By \cite{[CLR]},
\[
P(F)=\co(\vec{\omega})^{2}\oplus(\bigoplus\limits_{j\neq i}\co(\vec{x}_{i}-\vec{x}_{j})).
\]
Now consider the following pullback diagram
\[
\xymatrix{
  0 \ar[r]& \co(\vec{x}_{3}-\vec{x}_{4}) \ar@{=}[d] \ar[r] & C\ar@{}[dr]|-{\circlearrowleft} \ar[d]\ar[r] & P(F) \ar[d]\ar[r] & 0 \\
 0 \ar[r]  & \co(\vec{x}_{3}-\vec{x}_{4}) \ar[r] & E_{3}\oplus E_{4} \ar[r]& F \ar[r] & 0.}
\]
It is easy to see that for any $\vec{x}\in \mathbb{L}$ with $\delta(\vec{x})=0$,
\[
\Ext^{1}(\co(\vec{x}), \co(\vec{x}_{3}-\vec{x}_{4}))\neq 0 \ \text{if and only if} \
\vec{x}=\vec{x}_{1}-\vec{x}_{2}.
\]Thus in $\crd$,
\[C=E(\vec{x}_{1}-\vec{x}_{2}).
\]Similar to the proof of Lemma \ref{exchange triangle1}, there exists the following triangle in $\crc$
\[
F[-1]\to E(\vec{x}_{1}-\vec{x}_{2})\to E_{3}\oplus E_{4}\to F.
\]
\end{proof}

\subsection{Tilting object with rank-two bundles via mutation} This subsection proves Theorem \ref{I:thm1}.
\begin{proposition}\label{cluster tilting consisting of rank two bundles}
The object
\[
T_{\rk}=E\oplus
E(\vec{x}_{1}-\vec{x}_{2})\oplus(\bigoplus\limits_{i=3}^{4}
E_{i})\oplus(\bigoplus\limits_{j=1}^{2}E(\vec{\omega}+\vec{x}_{j}))
\]
is a cluster tilting object in $\crc$.
\end{proposition}

\begin{proof}By Proposition \ref{canonical cluster tilting},
\[T_{\can}=E\oplus(\bigoplus\limits_{ i=1}^{4}E_{i})\oplus F=E_1\oplus \overline{T}
\] is a cluster tilting object in $\crc$,
and the quiver of the endomorphism algebra $\End_{\crc}(T_{\can})$ has the following shape
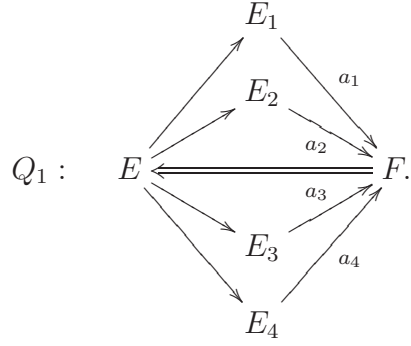
\[
\begin{array}{cccc}
 \xymatrix@-1pc
 {&&&E_1 \ar[rrdd]^{a_1} \\
  &&&E_2 \ar[rrd]_{a_2}\\
 Q_1:& E  \ar[rruu] \ar[rru]\ar[rrd] \ar[rrdd]  & &&&F.\ar@2{->}[llll]\\
  &&&E_3 \ar[rru]^{a_3} \\
  &&&E_4 \ar[rruu]_{a_4}
  }
  \end{array}
\]
By Lemma \ref{exchange triangle1}, there exists a triangle in $\crc$
\[E_1\xrightarrow{a_1}F\to E(\vec{\omega}+\vec{x}_1)\to E_1\st,
\]where $a_{1}: E_1\to F$ is the minimal left $\add\overline{T}$-approximation easily known from the above quiver.
Then by Theorem \ref{cluster mutation},
\[T_\star=E\oplus(\bigoplus\limits_{i=2}^{4} E_{i})\oplus F\oplus E(\vec\omega+\vec{x}_{1})
\]is a cluster tilting object in $\crc$. And the quiver of
$\End_{\crc}(T_\star)$ has the following shape
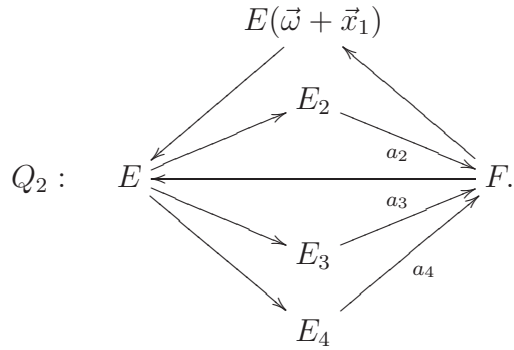
\[
\begin{array}{cccc}
& \xymatrix@-1pc
{&&&E(\vec\omega+\vec{x}_{1})\ar[lldd] \\
 & &&E_2 \ar[rrd]_{a_2} \\
 Q_2:& E \ar[rru] \ar[rrd] \ar[rrdd]  &&&&F. \ar[lluu] \ar[llll]\\
 & &&E_3 \ar[rru]^{a_3} \\
 & &&E_4 \ar[rruu]_{a_4}
  } \end{array}
\]
Write $T_\star$ as $T_\star=E_2\oplus\overline{T_\star}$. Similarly,
we obtain that $a_{2}: E_2\to F$ is the minimal left
$\add\overline{T_\star}$-approximation in $\crc$ and
\[T_{\star\star}=E\oplus(\bigoplus\limits_{i=3}^{4} E_{i})\oplus F\oplus (\bigoplus\limits_{j=1}^{2}E(\vec\omega+\vec{x}_{j}))
\]is a cluster tilting object. The quiver of the endomorphism algebra $\End_{\crc}(T_{\star\star})$
has the following shape
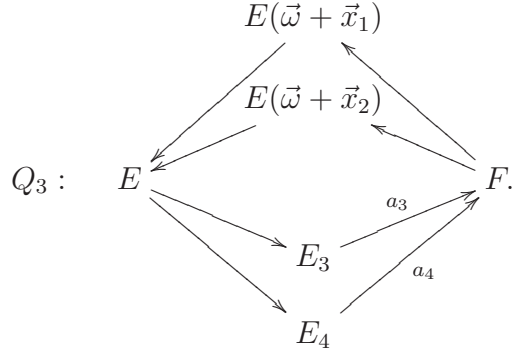
\[
\begin{array}{cccc}
& \xymatrix@-1pc
 {&&&E(\vec\omega+\vec{x}_{1})\ar[lldd]\\
  &&&E(\vec\omega+\vec{x}_{2}) \ar[lld] \\
  Q_3:&E \ar[rrd] \ar[rrdd]   & &&&F. \ar[lluu] \ar[llu]\\
  &&&E_3 \ar[rru]^{a_3} \\
  &&&E_4 \ar[rruu]_{a_4}
  }
\end{array}
\]Write $T_{\star\star}$ as $F\oplus \overline{T_{\star\star}}$.
By Lemma \ref{exchange triangle2}, there exists a triangle in $\crc$
\[
F[-1]\to E(\vec{x}_1-\vec{x}_2)\to E_3\oplus E_4\xrightarrow{(a_3, a_4)} F,
\]where $(a_3, a_4): E_3\oplus E_4\to F$ is the minimal right $\add \overline{T_{\star\star}}$-approximation easily known from the above quiver.
Hence $T_{\rk}$ is a cluster tilting object in $\crc$ and the quiver of the
endomorphism algebra $\End_{\crc}(T_{\rk})$ has the following shape
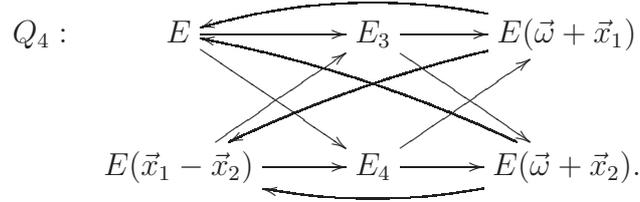
\[
\begin{array}{cccc}\\
Q_4:& \xymatrix@-1pc {
E\ar[rr]^{}\ar[rrdd]^>>>>>{}&&E_{3}\ar[rr]^{}\ar[rrdd]^>>>>>{}&&E(\vec{\omega}+\vec{x}_{1})\ar@/_1pc/[llll]\ar@/_/[ddllll] \\
  \\
E(\vec{x}_{1}
-\vec{x}_{2})\ar[rr]^{}\ar[rruu]_>>>>>{}&&E_{4}\ar[rr]^{}\ar[rruu]_>>>>>{}&&E(\vec{\omega}+\vec{x}_{2}).\ar@/^1pc/[llll]\ar@/_/[uullll]
}\\ \ \\
\end{array}
\]
\end{proof}

\begin{remark}\label{remark}During the proof, we obtain three cluster tilting objects reachable from $T_{\can}$: $T_{\star}, T_{\star\star}$, and $T_{\rk}$. The quivers of the associated endomorphism algebras are mutation-equivalent. In the cluster category $\crc$, all quivers mutation-equivalent to $Q=Q_1$ are obtained in this way. That is, for any cluster tilting object $T$ in $\crc$, the quiver of the endomorphism algebra $\End_{\crc}(T)$ has been listed in Example \ref{example}.
\end{remark}

\begin{theorem}\label{theorem2}
The object
\[
T_{\rk}=E\oplus
E(\vec{x}_{1}-\vec{x}_{2})\oplus(\bigoplus\limits_{i=3}^{4}
E_{i})\oplus(\bigoplus\limits_{j=1}^{2}E(\vec{\omega}+\vec{x}_{j}))
\]
is tilting in $\crd$, and the quiver of the endomorphism
algebra $\End_\crd(T_{\rk})$ has the shape\[
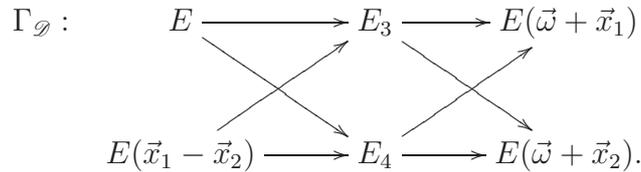
\begin{array}{cccc}
 \Gamma_\crd:& \xymatrix@-1pc
   {E\ar[rr]^{}\ar[rrdd]^>>>>>{}&&E_{3}\ar[rr]^{}\ar[rrdd]^>>>>>{}&&E(\vec{\omega}+\vec{x}_{1}) \\
    \\
    E(\vec{x}_{1}-\vec{x}_{2})\ar[rr]^{}\ar[rruu]_>>>>>{}&&E_{4}\ar[rr]^{}\ar[rruu]_>>>>>{}&&E(\vec{\omega}+\vec{x}_{2}).}
\end{array}
\]
\end{theorem}

\begin{proof} Noting that $T_{\rk}\in\crd_{\alpha^{-1}(1)}$, we have $\Phi_{\alpha^{-1}(1)}(T_{\rk})\in\coh\X$.
According to \cite[Proposition 2.3]{[BKL]}, tilting sheaves in $\coh\X$ coincide with cluster tilting objects in the cluster category. It follows from Proposition \ref{cluster tilting consisting of rank two bundles} that $\Phi_{\alpha^{-1}(1)}(T_{\rk})$ is a tilting sheaf. By Lemma \ref{Phi_q preserves tilting}, $T_{\rk}$ is a tilting object in $\crd$. It is easy to see that the quiver of the endomorphism algebra of $T_{\rk}$ has the claimed shape.
\end{proof}

\section{Classification of endomorphism algebras}\label{Sec Classification}
Let $\X$ be a weighted projective line with weight type $(2,2,2,2)$. This section is devoted to classifications of endomorphism algebras of tilting objects in $D^b(\coh\X)$ and endomorphism algebras of tilting sheaves in $\coh\X$.

\subsection{Tilting objects corresponding to a given cluster tilting object}
For a complete classification of endomorphism algebras of tilting objects in $D^b(\coh\X)$, we use the triangle equivalence \eqref{II:equ}
\[\underline{\mbox{vect}}\mathbb{X}\simeq D^b(\coh\X).
\]As before, we denote the stable category $\vect\X$ by $\crd$. Recall that the cluster category $\crc$ is the orbit category $\crc/G^{\mathbb{Z}}$ under the action of the cyclic group generated by the autoequivalence $G=\tau^{-1}\st$. The canonical projection $\pi: \crd\to\crc$ is a triangle functor.

Let $T=\bigoplus\limits_{j=1}^{6}T_{j}$ be a basic tilting object in $\crd$. Without loss of generality, from now on we always assume $\mu T_{i}\leq \mu T_{i+1}$ for $1\leq i\leq 5.$ The next lemma shows that we can obtain a series of tilting objects
from $T$.

\begin{lemma}\label{lemma7}
For $1\leq i\leq 5$, the object
\[
(\bigoplus\limits_{j=1}^{i}GT_{j})\oplus(\bigoplus\limits_{j=i+1}^{6}T_{j})
\]is tilting in $\crd$.
\end{lemma}

\begin{proof}We only prove the object
\[
T'=GT_{1}\oplus(\bigoplus\limits_{i=2}^{6}T_{i})
\] is tiling in $\crd$, the others are similar.

The following two equalities,
\[
\crd(GT_{1}, T_{i}[n]) =D\crd(T_{i}[n-2], T_{1})=0\ \text{for any}\
n\in\mathbb{Z}
\]and
\[\crd(T_{i}, GT_{1}[n])=D\crd(T_{1}[n], T_{i})=0\ \text{for any}\ n\neq
0,
\]
imply that $T'$ is extension-free. Note that\[
 \crd(GT_{1},
T_{1}[2])=D\crd(T_{1}, T_{1})\neq
0,\]
$T_{1}$ is in the thick subcategory generated by $T'$. Then $T'$ is tilting provided that $T$ is tilting in $\crd$.
\end{proof}

\begin{proposition}\label{theorem3}The image of $T$ under the projection $\pi$ is cluster tilting in $\crc$.
\end{proposition}
\begin{proof}Proposition \ref{proposition4} implies that
\[
\mu T_{6}\leq \mu (T_{1}[1]).
\]There are two cases to consider.
\begin{itemize}
\item[\emph{Case 1}]: $\mu T_{6}<\mu( T_{1}[1])$.
Then all the indecomposable direct summands are of slopes in the interval $(\mu(T_{6}[-1]), \mu (T_{6})]$.
Hence $\pi(T)$ is a cluster tilting object in
$\crc$.
\item[\emph{Case 2}]: $\mu T_{6}=\mu(T_{1}[1])$. Let $i$ be the largest index satisfying $\mu T_{1}=\mu T_{i}$.
Lemma \ref{lemma7} implies that
\[
T''=(\bigoplus\limits_{j=1}^{i}GT_{j})\oplus(\bigoplus\limits_{j=i+1}^{6}T_{j})
\]is tilting in $\crd$.
Clearly, the slope of each indecomposable direct summand of $T''$ is in the interval
$(\mu T_{1}, \mu (T_{1}[1])]$.
Then $\pi(T'')$ is a cluster tilting object in $\crc$.
Note that $T$ and $T''$ have the same image in $\crc$. We get what we want.
\end{itemize}
\end{proof}

Next we describe all the tilting objects
corresponding to a given cluster tilting object.
A \emph{lifting} of $\pi(T)$ to $\crd$ is an object $X$ in $\crd$ with
$\pi(X)=\pi(T)$. Obviously, $T$ is a lifting of $\pi(T)$, and any
other lifting has the form
\[
\bigoplus\limits_{i=1}^{6}G^{k_{i}}T_{i},
\ \text{where}\ k_{i}\in\mathbb{Z}.
\]

\begin{theorem}\label{theorem4}Let $T'=\bigoplus\limits_{i=1}^{6}G^{k_{i}}T_{i}$ be a lifting of $\pi(T)$.
Then $T'$ is tilting in $\crd$ if and only if $k_{i}\geq k_{j}\geq
k_{i}-1$ whence $\mu T_{i}<\mu T_{j}$.
\end{theorem}
\begin{proof}Assume $T'$ is tilting in $\crd$ and $\mu T_{i}<\mu T_{j}$. Then Corollary \ref{proposition3} implies that
\[
\crd(T_{i}, \tau T_{j}\oplus T_{j})\neq 0.
\]
Notice that
\[
\crd(G^{k_{i}}T_{i}, G^{k_{j}}T_{j}[k_{i}-k_{j}]) =\crd(T_{i},
\tau^{k_{i}-k_{j}}T_{j})
\]
and
\[\crd(G^{k_{j}}T_{j}, G^{k_{i}}T_{i}[k_{j}-k_{i}+1])
=\crd(T_{j}, \tau^{k_{j}-k_{i}}T_{i}[1]) =D\crd(T_{i},
\tau^{k_{i}-k_{j}+1}T_{j}).
\]Hence $T'$ is extension-free implies that
either $k_{i}-k_{j}=0$ or $k_{j}-k_{i}+1=0$, that is, $k_{i}\geq
k_{j}\geq k_{i}-1$.

Conversely, assume $\mu T_{i}< \mu T_{j}$ implies $k_{i}\geq
k_{j}\geq k_{i}-1$. Arrange the indecomposables $T_i$ with the same
slope to ensure
\[
k_{1}\geq k_{2}\geq\cdots\geq k_{6}\geq k_{1}-1.
\]
If $k_{1}= k_{2} =\cdots = k_{6}$,
then $T'=G^{k_{6}}T$ is a tilting object in $\crd$.
If else, there exists some $1\leq l\leq 5$,
such that
\[
k_{1}=\cdots=k_{l}>k_{l+1}=\cdots=k_{6}=k_1-1.
\]So
\[
T'=G^{k_{6}}(G(T_{1}\oplus\cdots\oplus T_{l})\oplus(T_{l+1}\oplus\cdots\oplus T_{6})).
\]By Lemma \ref{lemma7}, $T'$ is a tilting object in $\crd$.
\end{proof}

\begin{corollary}\label{proposition5}
Let $T'=\bigoplus\limits_{i=1}^{6}G^{k_{i}}T_{i}$ be a lifting of $\pi(T)$. If $T'$ is a tilting object in $\crd$, then
$\mu(G^{k_{i}}T_{i})\in(q, \alpha(q)]$ for any $i$ and some $ q\in \mathbb{Q}$ if and only if
$k_{i}=k_{j} \ \text{whence}\ \mu T_{i}= \mu T_{j}.$
\end{corollary}
\begin{proof}Assume that all the slopes $\mu(G^{k_{i}}T_{i})$ belong to $(q, \alpha(q)]$
and $\mu T_{i}=\mu T_{j}$. If $k_{i}\neq k_{j}$, we assume
$k_{i}>k_{j}$ without loss of generality. Then
\[
\mu(G^{k_{i}}T_{i})=
\mu(T_{i}[k_{i}])\geq\mu(T_{j}[k_{j}+1])=\alpha(\mu(T_{j}[k_{j}]))=\alpha(\mu(G^{k_{j}}T_{j})),
\]
which gives a contradiction.

On the contrary,
by Theorem \ref{theorem4},
the tilting object $T'$ has the form
\[
T'=G^{k_{6}}(G(T_{1}\oplus\cdots\oplus
T_{l})\oplus(T_{l+1}\oplus\cdots\oplus T_{6}))\ \text{for some}\ l.
\]Since $\mu T_{i}= \mu T_{j}$ implies $k_{i}= k_{j}$, we have $\mu T_{l}<\mu T_{l+1}$. Hence for any $1\leq i\leq
6$, we have
\[
\mu(G^{k_{i}}T_{i})\in [\mu(G^{k_{6}}T_{l+1}),
\mu(G^{k_{6}+1}T_{l})] \subseteq(\mu(G^{k_{6}}T_{l}),
\mu(G^{k_{6}+1}T_{l})].
\]
We are done.
\end{proof}

\subsection{Endomorphism algebras of tilting objects in $D^b(\coh\X)$}

In this subsection, we prove Theorem \ref{I:thm2}. As before, let
$T=\bigoplus\limits_{j=1}^{6}T_{j}$ be a basic tilting object in $\crd$
where $T_{i}\in \crd_q$ for some $q\in \mathbb{Q}$. Let
$\Gamma_\crc$ be the quiver of the endomorphism algebra
$\End_\crc(T)$ and $\Gamma_\crd$ be the one of $\End_\crd(T)$.
\begin{lemma}\label{lemma8}
For any $i\neq j$,
\begin{enumerate}
\item[(1)] $\crd(T_{i}, T_{j})\neq 0$ if and only if $\mu T_{i}<\mu T_{j}$;
\item[(2)] $\crc(T_{i}, T_{j})=0$ if and only if $\mu T_{i} =\mu T_{j}$.
\end{enumerate}
\end{lemma}
\begin{proof} (1) By \cite{[CLR]},
the indecomposable direct summands of $T$ lie in the bottom of tubes
of rank two, and they are orthogonal to each other if they have the
same slope. Hence $\crd(T_{i}, T_{j})\neq 0$ implies $\mu T_{i}< \mu
T_{j}$.

Conversely, by Corollary
\ref{proposition3},
$\mu T_{i} <\mu T_{j}$ implies that
\[
\crd(T_{i}, T_{j}\oplus\tau T_{j})\neq 0.
\]
But
\[
\crd(T_{i}, \tau T_{j})=D\crd(T_{j}, T_{i}[1])=0.
\]Thus $\crd(T_{i}, T_{j})\neq 0$.

(2) Let $\crc(T_{i}, T_{j})=0$. If $\mu
T_{i} \neq \mu T_{j}$,  without loss of generality we assume $\mu
T_{i} <\mu T_{j}$. Then $\crd(T_{i},
T_{j})\neq 0$ by (1). It follows that $\crc(T_{i}, T_{j})\neq 0$, which is a
contradiction.

On the contrary, $\mu T_{i} =\mu T_{j}$ implies that $(T_{i}, T_{j})$ is an orthogonal pair (\confer \cite{[CLR]}).
Hence
\[
\crc(T_{i}, T_{j}) =\crd(T_{i}, T_{j})\oplus\crd(T_{i},
GT_{j})=\crd(T_{i}, T_{j})\oplus D\crd(T_{j}, T_{i}) =0.
\]
\end{proof}

Let $s_{i}: T_{i}\to U_{i}$ be the minimal left $\add(T\backslash
T_{i})$-approximation of $T_{i}$ in $\crc$.

\begin{lemma}\label{lemma9}
Assume $\Gamma_\crc$ has no 2-cycles. Let $T_m, T_n$ be two indecomposable direct summands of $T$ satisfying the following two conditions:
\begin{itemize}
\item[(a)] $s_m$ and $s_n$ map to the same object $U$;
\item[(b)] for each indecomposable direct summand $T_i$ of $U$,
\[
\dim\crc(T_m, T_i)=1=\dim\crc(T_n, T_i);
\]
\end{itemize}
then $\mu T_m=\mu T_n$.
\end{lemma}
\begin{proof}For contradiction,
we assume $\mu T_m<\mu T_n$. Then
$\crd(T_m, T_n)\neq 0$ by Lemma \ref{lemma8} (1). So there
exists a path $\rho$ from $T_m$ to $T_n$ in $\Gamma_\crd$ and
then in $\Gamma_\crc$. By condition (a), the length of $\rho$ is
greater than one. Hence there exists at least one indecomposable
direct summand $T_i$ of $U$, such that
\[
\mu T_m <\mu T_i <\mu T_n .
\]
Furthermore, we claim that for any indecomposable summand $T_j$ of
$U$,
\begin{equation}\label{assumption result 6}
\mu T_m <\mu T_j <\mu T_n .
\end{equation} In fact, if $\mu T_j \geq\mu T_n $ for some $j$,
according to condition (b), we get
\[
\dim\crd(T_m, T_j)=1=\dim\crd(T_n, T_j).
\]
Then by condition (a), the composition $T_m\to T_n\to T_j$ vanishes, which
induces an arrow from $T_j$ to $T_m$ in $\Gamma_\crc$ since $\crc(T,
T)$ can be explained as a trivial-extension and then a
relation-extension algebra of $\crd(T, T)$ (\confer \cite{[Z], [ABS]}).
Hence a 2-cycle between $T_m$ and $T_j$ appears in $\Gamma_\crc$ which is a
contradiction. If $\mu T_j \leq\mu T_m $ for some $j$, then
$\crd(T_j, T_i)\neq 0$ by Lemma \ref{lemma8} (1). Moreover, according
to condition (b),
\[
\dim\crd(G^{-1}T_n, T_j)=1=\dim\crd(G^{-1}T_n, T_i).
\]
Similar arguments show that a 2-cycle between $T_n$ and $T_i$
appears in $\Gamma_\crc$, which is a contradiction. Thus the claim
(\ref{assumption result 6}) holds. It follows that
\[
\crc(T_n, U) =\crd(T_n, GU).
\]
Hence the approximation $s_n: T_n\to U$ in $\crc$ lifts to a triangle in $\crd$
\[
\varepsilon: T_n\xrightarrow{s_{n}} GU\to T_n^{\ast}\to T_n[1].
\]Applying $\crd(T_m,-)$ to $\varepsilon$,
we obtain
\begin{equation}\label{assumption result 7}
\crd(T_m,T_n^{\ast}[-1])\neq 0.
\end{equation}
On the other hand, $\pi(\varepsilon)$ is the following triangle in $\crc$:
\[
\overline{\varepsilon}: T_n \xrightarrow{s_{n}} U\to T_n^{\ast}\to T_n[1].
\]That is, $T_n^{\ast}$ is a complement of the almost complete cluster
tilting object $T\backslash T_n$. But $[2]=G^2$ in $\crd$,
\[
\crc(T_m, T_n^{\ast}[-1])=\crc(T_m, T_n^{\ast}[1])=0,
\] which gives a contradiction to (\ref{assumption result 7}). This finishes the proof.
\end{proof}

We are now in the position to prove Theorem \ref{I:thm2}.
\begin{proof} [\textbf {Proof of Theorem \ref{I:thm2}}]Let $\Lambda$ be a finite dimensional $k$-algebra. Assume $\Lambda$ is the endomorphism algebra of some tilting object $T_c$ in $D^{b}(\coh\X)$. We regard $T_c$ as a
tilting object in $\crd$ and $\Lambda=\End_\crd(T_c)$. By Proposition
\ref{theorem3}, $\pi(T_c)$ is a cluster tilting object in $\crc$.
Hence the quiver $\Gamma$ of the endomorphism algebra
$\End_\crc(\pi(T_c))$ belongs to list (\ref{quiver mutation
classes}) according to Remark \ref{remark}. We then suppose $\pi(T_c)=\pi(T)$, where
$T=\bigoplus\limits_{i=1}^{6}T_{i}\in\crd_q$ for some $q\in\mathbb{Q}$. As before, we
assume $\mu T_{i}\leq \mu T_{i+1}$ for $1\leq i\leq 5$.

If $\Gamma=Q_{1}$, then by Lemmas \ref{lemma8} and \ref{lemma9}, we can assume
\[
\mu T_{1}< \mu T_{2}=\mu T_{3}=\mu T_{4}=\mu T_{5}<\mu T_{6}.
\]
By Theorem \ref{theorem4}, $T_c$ has the form(under the equivalence
$G^{k_{6}}$)
\[
\bigoplus\limits_{i=1}^{6}T_{i}\ \text{or}\ GT_{1}\oplus(\bigoplus\limits_{i=2}^{5}G^{k_{i}}T_{i})\oplus
T_{6},
\]where $k_{i}=0$ or $1$ for $2\leq i\leq 5$.
For some choice of the representatives for the arrows, $\Lambda$ is
isomorphic to $B_{11}$, $B_{12}$, $A_{13}$, $A_{14}$, $ A_{15}$ or
$B_{13}$ in List \ref{list}.

Similarly, one can prove that if $\Gamma=Q_{2}$, then $\Lambda$ is
isomorphic to $B_{21}$, $B_{22}$, $A_{23}$, $A_{24}$, $ B_{23}$ or
$B_{24}$; if $\Gamma=Q_{3}$, then $\Sigma$ is isomorphic to
$B_{31}$, $B_{32}$, or $A_{33}$; if $\Gamma=Q_{4}$, then $\Lambda$ is
isomorphic to $B_{41}$ or $A_{42}$ in List \ref{list}.

Conversely, we claim that each algebra in List \ref{list} can be realized by
a tilting object in $\crd$. In fact, by \cite[Theorem 6.2]{[CLR]} we
know that the tilting object $T_{\can}$ gives a realization of the algebra $B_{11}$. Combining with
Theorem \ref{theorem4} and using the method similar to the proof of
\cite[Theorem 6.2]{[CLR]}, it is easy to check that by replacing the
summand $E$ with $GE$ for $T_{\can}$, we get a realization of
$B_{12}$, i.e. the object $GE\oplus(\bigoplus\limits_{
i=1}^{4}E_{i})\oplus F$ is tilting in $\crd$ with endomorphism
algebra $B_{12}$. Similarly, for $j=1,2,3,4$, by replacing the
summand $E\oplus (\bigoplus_{i=1}^{j} E_i)$ of $T_{\can}$ with its
image under the functor $G$, we get realizations of $A_{13}, A_{14},
A_{15}$ and $B_{13}$ respectively. Analogously, using the tilting
objects $T_{\star}, T_{\star\star}$ and $T_{\rk}$ appeared in
Theorem \ref{theorem2} and combining with Theorem \ref{theorem4}, we
can get realizations of all the other algebras in List \ref{list}. This
finishes the proof.
\end{proof}

\begin{remark} In List \ref{list}, we also provide another realization for each algebra by tilting complexes in $D^b(\coh\X)$ with line bundles and simples sheaves (up to suspension shift).

\end{remark}
\subsection{Endomorphism algebras of tilting sheaves in $\coh\X$ }

This subsection is devoted to proving Theorem \ref{I:thm3}.

\begin{proof}[\textbf{Proof of Theorem \ref{I:thm3}}]Let $\Lambda'$ be a finite dimensional $k$-algebra. Assume $\Lambda'$ is the endomorphism algebra of some tilting object in $\coh\X$. Since $\coh\X$ is equivalent to some interval category $\crd_q$, $\Lambda'$ can be viewed as
the endomorphism algebra of a tilting object in $\crd_q$ for some $q$, which corresponds to a cluster tilting object in $\crc$. Assume $T=\bigoplus\limits_{i=1}^{6}T_{i}\in\crd_q$ is a tilting object with $\mu T_{i}\leq \mu T_{i+1}$ for $1\leq i\leq 5$. The quiver $\Gamma$ of the endomorphism algebra
$\End_\crc(\pi(T))$ belongs to list (\ref{quiver mutation
classes}). If $\Gamma=Q_{1}$, by Lemmas \ref{lemma8} and
\ref{lemma9}, we assume
\[
\mu T_{1}<\mu T_{2}=\mu T_{3}=\mu T_{4}=\mu T_{5}<\mu T_{6}.
\]Then according to Corollary
\ref{proposition5}, a lifting of $\pi(T)$ in $\crd_{q'}$ (for some $q'$) has one of the
following forms(under the equivalence $G^{k_{6}}$):
\[
\bigoplus\limits_{i=1}^{6}T_{i},\ \ \ \
GT_{1}\oplus(\bigoplus\limits_{i=2}^{6}T_{i}), \ \ \ \
(\bigoplus\limits_{i=1}^{5}GT_{i})\oplus T_{6}.
\]For some choice of the representatives for the arrows,
we obtain that the endomorphism algebras $\Lambda'$ is isomorphic to
$B_{11}$, $ B_{12}$ or $B_{13}$ in List \ref{list}.

Similarly, one can prove that if $\Gamma=Q_{2}$, then $\Lambda'$ is
isomorphic to  $B_{21}$, $B_{22}$, $ B_{23}$ or $B_{24}$; if
$\Gamma=Q_{3}$, then $\Lambda'$ is isomorphic to $B_{31}$ or $B_{32}$;
if $\Gamma=Q_{4}$, then $\Lambda'$ is isomorphic to $B_{41}$.

Conversely, the tilting object corresponding to $B_{ij}$ appeared in the proof of Theorem \ref{I:thm2}
gives a tilting sheaf we need.
\end{proof}


\vspace{0.2cm} \noindent{\bf Acknowledgments.}
{This work was
partially supported by the National Natural Science Foundation of China (Grant No. 11571286, 11871404, 11801473) and the Fundamental Research Funds for the Central Universities of China (Grant No. 20720180002, 20720180006). The authors would like to thank X.W. Chen, B. M. Deng and H. Lenzing for useful comments.}




\begin{appendix}

\section{\label{sec:compint}}

In this appendix, we give a complete list of endomorphism algebras of tilting complexes in $D^{b}(\coh\X)$ by quivers with relations, together with a realization by tilting complexes with line bundles and simples sheaves (up to suspension shift).

\begin{List}\label{list}{{\bf Endomorphism algebras of tilting complexes in $D^{b}(\coh\X)$}}
\end{List}

{\tiny{\begin{center}
$\begin{array}{ccccclccc} \\
\bf{algebra} && \bf{quiver} & \bf{relations}& \bf{a \ realization\ by\ complex} \\[-4em]
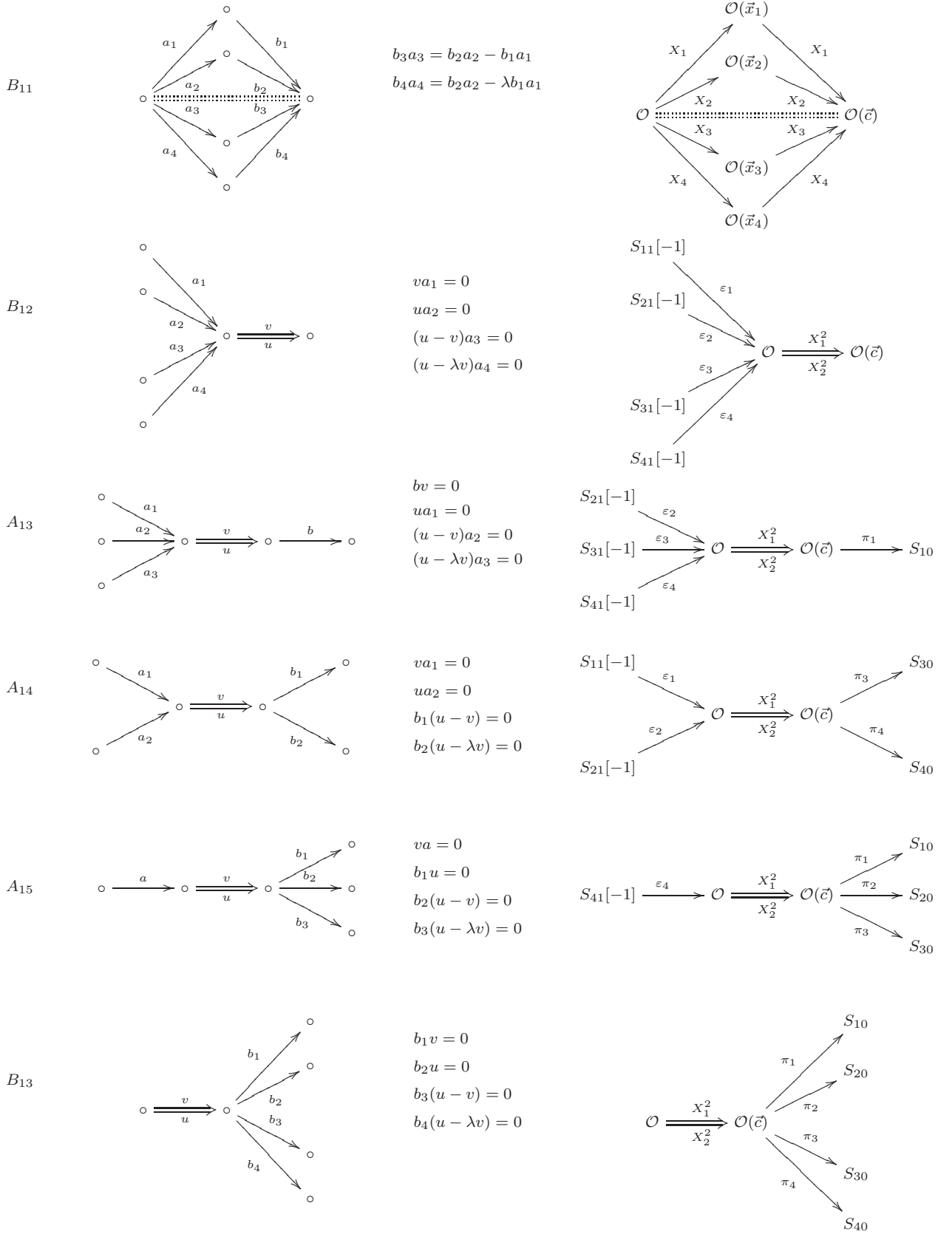
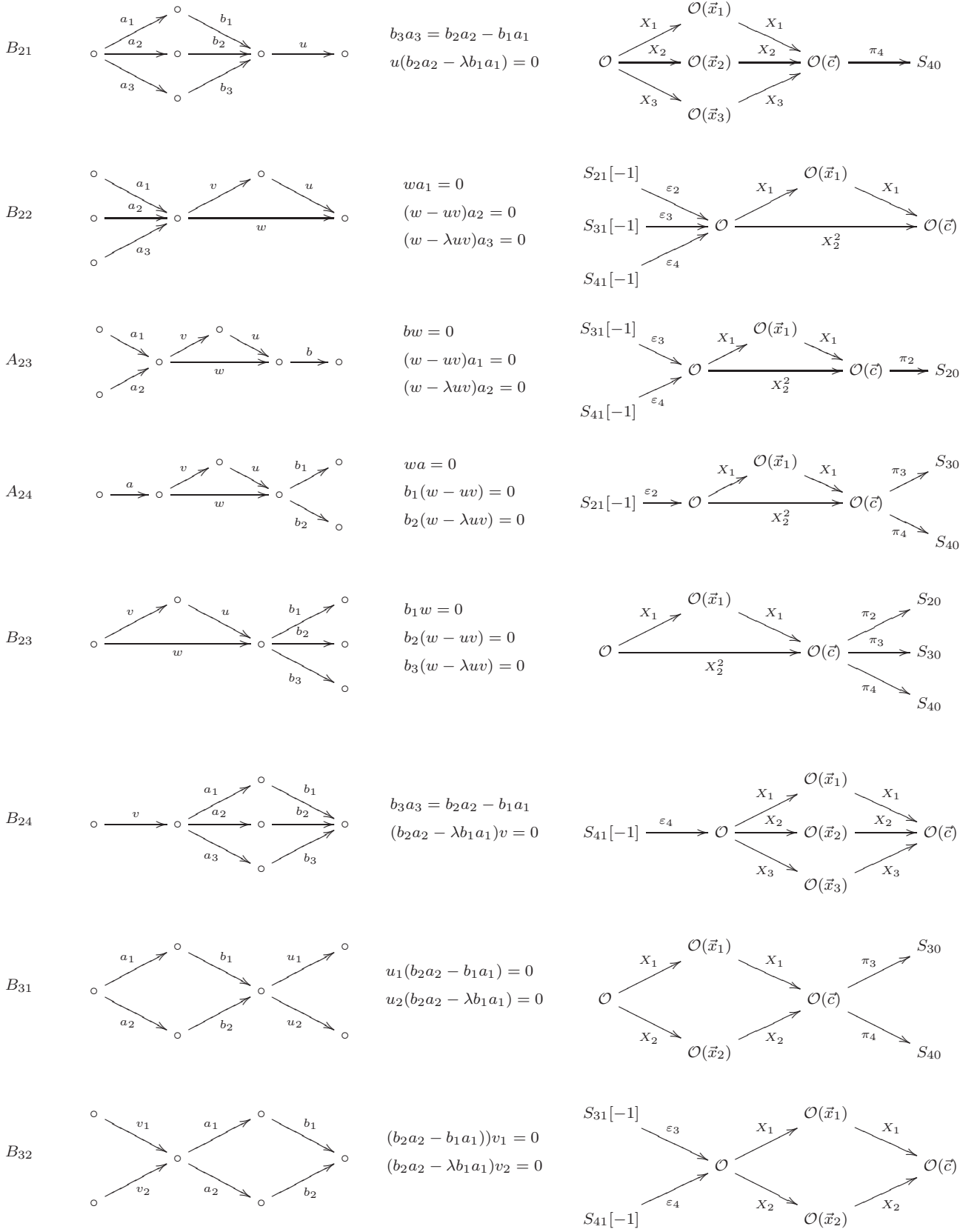
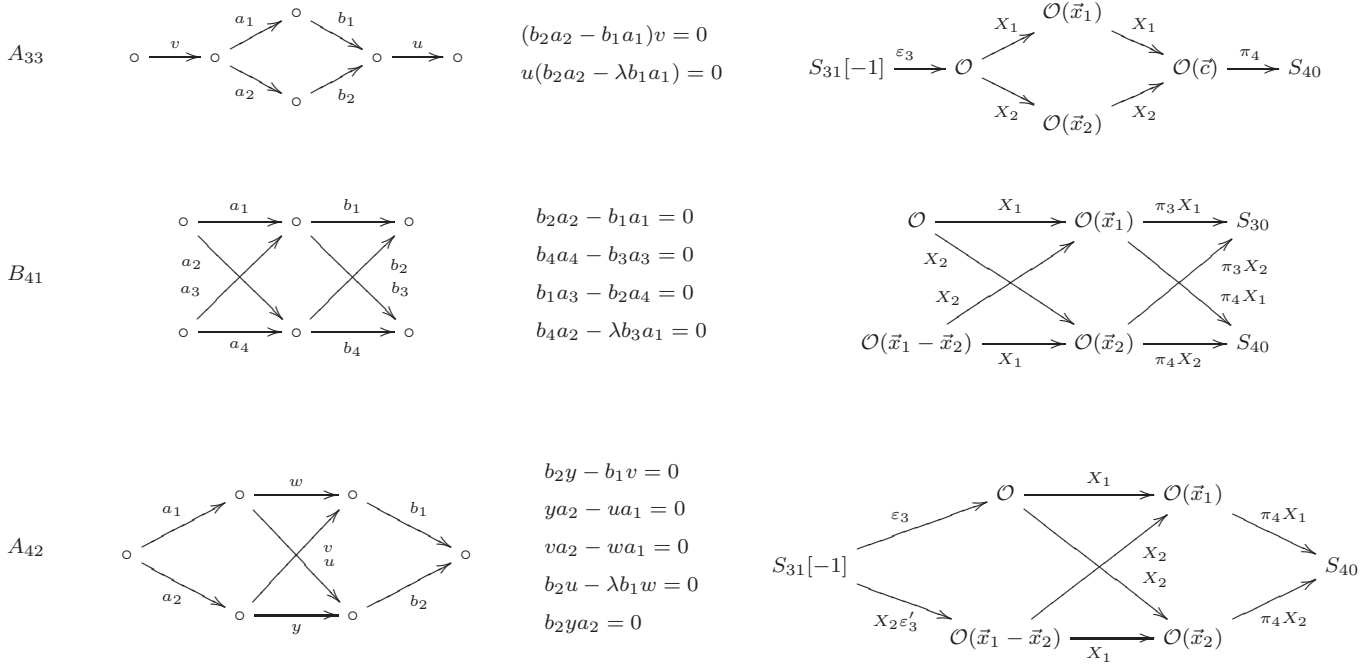
 \begin{array}{l}\\[8em]
  B_{11}
 \end{array}
 &&
  {\xymatrix@-1pc{&&\circ \ar[rrdd]^{b_{1}} \\
  &&\circ \ar[rrd]_{b_{2}} \\
  \circ \ar[rrd]^{a_{3}} \ar[rrdd]_{a_{4}}  \ar[rru]_{a_{2}} \ar[rruu]^{a_{1}} & &&&\circ\ar@2{.}[llll]_{}\\
  &&\circ \ar[rru]^{b_{3}} \\
  &&\circ \ar[rruu]_{b_{4}}}}
 &
  {\begin{array}{l}\\[6em]
b_{3}a_{3}=b_{2}a_{2}-b_{1}a_{1}\\[0.5em]
b_{4}a_{4}=b_{2}a_{2}-\lambda b_{1}a_{1}
\end{array} }
&
{\xymatrix@-1pc{&&\co(\vx_1) \ar[rrdd]^{X_{1}} \\
  &&\co(\vx_2) \ar[rrd]_{X_{2}} \\
  \co \ar[rrd]^{X_{3}} \ar[rrdd]_{X_{4}}  \ar[rru]_{X_{2}} \ar[rruu]^{X_{1}} & &&&\co(\vc)\ar@2{.}[llll]_{}\\
  &&\co(\vx_3) \ar[rru]^{X_{3}} \\
  &&\co(\vx_4) \ar[rruu]_{X_{4}}}}\\[-7em]

\begin{array}{l}\\[6em]
  B_{12}
 \end{array}
 &&
  {\xymatrix@-1pc{\circ \ar[rrdd]^{a_{1}} \\
  \circ \ar[rrd]_{a_{2}} \\
  &&\circ \ar@{=>}[rr]^{v}_{u} & &\circ\\
  \circ\ar[rru]^{a_{3}} \\
  \circ \ar[rruu]_{a_{4}} }}
&
  {\begin{array}{l}\\[8em]
  va_{1}=0\\[0.5em]
   ua_{2}=0\\[0.5em]
    (u-v)a_{3}=0\\[0.5em]
(u-\lambda v)a_{4}=0
\end{array} }
&
  {\xymatrix@-1pc{S_{11}[-1]  \ar[rrdd]^{\eps_{1}} \\
  S_{21}[-1]  \ar[rrd]_{\eps_{2}} \\
  &&\co \ar@{=>}[rr]^{X^2_{1}}_{X^2_{2}} & &\co(\vc)\\
  S_{31}[-1]  \ar[rru]^{\eps_{3}} \\
  S_{41}[-1]  \ar[rruu]_{\eps_{4}} }}\\[-3em]

\begin{array}{l}\\[2em]
  A_{13}
 \end{array}
 &&
  {\xymatrix@-1pc{\circ \ar[rrd]^{a_{1}} \\
  \circ \ar[rr]^{a_{2}}
  &&\circ \ar@{=>}[rr]^{v}_{u} & &\circ\ar[rr]^{b}&& \circ\\
  \circ \ar[rru]_{a_{3}} }}
&
  {\begin{array}{ll}
 \\[2em]
bv=0\\[0.3em]
ua_{1}=0\\[0.3em]
(u-v)a_{2}=0\\[0.3em]
(u-\lambda v)a_{3}=0
\end{array} }
&
  {\xymatrix@-1pc{S_{21}[-1] \ar[rrd]^{\eps_{2}} \\
  S_{31}[-1] \ar[rr]^{\eps_{3}}
  &&\co \ar@{=>}[rr]^{X_1^2}_{X^2_{2}} & &\co(\vc)\ar[rr]^{\pi_{1}}&& S_{10}\\
  S_{41}[-1] \ar[rru]_{\eps_{4}} }}
  \\[-3em]

\begin{array}{l}\\[2em]
  A_{14}
 \end{array}\hspace{-2mm}
 &&\hspace{-2mm}
  {\xymatrix@-1pc{ \circ \ar[rrd]^{a_{1}}&&&&&&\circ \\
  &&\circ \ar@{=>}[rr]^{v}_{u} & &\circ\ar[rru]^{b_{1}}\ar[rrd]_{b_{2}}\\
  \circ \ar[rru]_{a_{2}}&&&&&&\circ}}
&
  {\begin{array}{ll}
 \\ [4em]
va_{1}=0 \\[0.5em]
ua_{2}=0 \\[0.5em]
b_1(u-v)=0\\[0.5em]
b_2(u-\lambda v)=0
\end{array} }
&
  {\xymatrix@-1pc{ S_{11}[-1]  \ar[rrd]^{\eps_{1}}&&&&&&S_{30}  \\
  &&\co \ar@{=>}[rr]^{X^2_{1}}_{X^2_{2}} & &\co(\vc)\ar[rru]^{\pi_{3}}\ar[rrd]^{\pi_{4}}\\
  S_{21}[-1]  \ar[rru]^{\eps_{2}}&&&&&&S_{40}  }}
\\[-2em]

 \begin{array}{l}\\[4em]
  A_{15}
 \end{array}
 &&
  {\xymatrix@-1pc{
  &&&&&&\circ  \\
\circ \ar[rr]^{a}&& \circ\ar@{=>}[rr]^{v}_{u} && \circ \ar[rru]^{b_{1}} \ar[rr]^{b_{2}} \ar[rrd]_{b_{3}}&&\circ\\
    &&&&&&\circ
    }}
&
  {\begin{array}{l}
  \\ [4em]
  va=0 \\ [0.5em]
b_{1}u=0 \\ [0.5em]
b_{2}(u-v)=0 \\ [0.5em]
b_{3}(u-\lambda v)=0
\end{array} }
&
  {\xymatrix@-1pc{
  &&&&&&S_{10} \\
S_{41}[-1] \ar[rr]^{\eps_{4}}&& \co\ar@{=>}[rr]^{X^2_{1}}_{X^2_{2}} && \co(\vc) \ar[rru]^{\pi_{1}} \ar[rr]^{\pi_{2}} \ar[rrd]_{\pi_{3}}&&S_{20}\\
    &&&&&&S_{30}
    }}\\[-3em]

\begin{array}{l}\\[6em]
  B_{13}
 \end{array}
 &&
  {\xymatrix@-1pc{
  &&&&\circ \\
  &&&&\circ  \\
  \circ \ar@{=>}[rr]^{v}_{u} & &\circ \ar[rruu]^{b_{1}} \ar[rru]_{b_{2}} \ar[rrd]^{b_{3}} \ar[rrdd]_{b_{4}}\\
    &&&&\circ\\
    &&&&\circ
    }}
&
  {\begin{array}{ll}\\ [6em]
  b_{1}v=0\\ [0.5em]
  b_{2}u=0\\ [0.5em]
  b_{3}(u-v)=0\\ [0.5em]
  b_{4}(u-\lambda v)=0
\end{array} }
&
  {\xymatrix@-1pc{
  &&&&S_{10}  \\
  &&&&S_{20}   \\
  \co \ar@{=>}[rr]^{X^2_{1}}_{X^2_{2}} & &\co(\vc) \ar[rruu]^{\pi_{1}} \ar[rru]_{\pi_{2}} \ar[rrd]^{\pi_{3}} \ar[rrdd]_{\pi_{4}}\\
    &&&&S_{30}   \\
    &&&&S_{40}
    }}\\[-2em]

\end{array}$

$\begin{array}{cccccl}\\

\begin{array}{l}\\[3.5em]
  B_{21}
 \end{array}
&&
  {\xymatrix@-1pc{
  &&\circ\ar[rrd]^{b_{1}}&& \\
  \circ \ar[rru]^{a_{1}}\ar[rr]^{a_{2}} \ar[rrd]_{a_{3}}&&\circ\ar[rr]^{b_{2}}&&\circ\ar[rr]^{u} &&\circ\\
    &&\circ\ar[rru]_{b_{3}}
    }}
&
  {\begin{array}{l}\\ [3.5em]
b_{3}a_{3}=b_{2}a_{2}-b_{1}a_{1}\\ [0.5em] u(b_{2}a_{2}-\lambda
b_{1}a_{1})=0
\end{array} }
&
  {\xymatrix@-1pc{
  &&\co(\vx_1)\ar[rrd]^{X_{1}}&& \\
  \co \ar[rru]^{X_{1}}\ar[rr]^{X_{2}} \ar[rrd]_{X_{3}}&&\co(\vx_2)\ar[rr]^{X_{2}}&&\co(\vc) \ar[rr]^{\pi_{4}} &&S_{40}\\
    &&\co(\vx_3)\ar[rru]_{X_{3}}
    }}\\[-2em]

\begin{array}{l}\\[3.5em]
  B_{22}
 \end{array}
 &&
  {\xymatrix@-1pc{
  \circ\ar[rrd]^{a_{1}}&& &&\circ\ar[rrd]^{u} \\
  \circ \ar[rr]^{a_{2}}&&\circ\ar[rrrr]_{w} \ar[rru]^{v}&&&&\circ \\
    \circ\ar[rru]_{a_{3}}
    }}
&
  {\begin{array}{ll}\\ [3.5em]
wa_{1}=0\\ [0.5em]
(w-uv)a_{2}=0\\ [0.5em]
(w-\lambda uv)a_{3}=0
\end{array} }
&
  {\xymatrix@-1pc{
  S_{21}[-1]\ar[rrd]^{\eps_{2}}&& &&\co(\vx_1)\ar[rrd]^{X_{1}} \\
  S_{31}[-1] \ar[rr]^{\eps_{3}}&&\co\ar[rrrr]_{X_2^2} \ar[rru]^{X_{1}}&&&&\co(\vc) \\
    S_{41}[-1]\ar[rru]_{\eps_{4}}
    }}\\[-2em]

\begin{array}{l}\\[2.5em]
  A_{23}
 \end{array}
 &&
  {\xymatrix@-1.5pc{
  \circ\ar[rrd]^{a_{1}}&& &&\circ\ar[rrd]^{u} \\
  &&\circ\ar[rrrr]_{w} \ar[rru]^{v}&&&&\circ \ar[rr]^{b}&&  \circ\\
    \circ\ar[rru]_{a_{2}}
    }}
&
  {\begin{array}{ll}
  \\ [2.5em]
bw=0\\ [0.5em]
(w-uv)a_{1}=0\\ [0.5em]
(w-\lambda uv)a_{2}=0
\end{array} }
&
  {\xymatrix@-1.5pc{
  S_{31}[-1]\ar[rrd]^{\eps_{3}}&& &&\co(\vx_1)\ar[rrd]^{X_{1}} \\
  &&\co\ar[rrrr]_{X^2_2} \ar[rru]^{X_{1}}&&&&\co(\vc)  \ar[rr]^{\pi_{2}}&&  S_{20}\\
    S_{41}[-1]\ar[rru]_{\eps_{4}}
    }}\\[-2em]

\begin{array}{l}\\[2.5em]
  A_{24}
 \end{array}
 &&
  {\xymatrix@-1.5pc{
  && &&\circ\ar[rrd]^{u} && &&\circ\\
  \circ\ar[rr]^{a}&&\circ \ar[rru]^{v}\ar[rrrr]_{w} &&&&\circ\ar[rru]^{b_{1}} \ar[rrd]_{b_{2}}\\
   && &&&& &&\circ
    }}
&
  {\begin{array}{l}
  \\ [2.5em]
 wa=0 \\ [0.5em]
b_1(w-uv)=0\\ [0.5em]
b_2(w-\lambda uv)=0
\end{array} }
&
  {\xymatrix@-1.5pc{
  && &&\co(\vx_1)\ar[rrd]^{X_1} && && S_{30}\\
  S_{21}[-1]\ar[rr]^{\eps_2}&&\co \ar[rru]^{X_1}\ar[rrrr]_{X_2^2} &&&&\co(\vc)\ar[rru]^{\pi_{3}} \ar[rrd]_{\pi_{4}}\\
   && &&&& &&S_{40}
    }}\\[-2em]

\begin{array}{l}\\[3.5em]
  B_{23}
 \end{array}
 &&
   {\xymatrix@-1pc{
  &&\circ\ar[rrd]^{u}&&&& \circ \\
  \circ\ar[rrrr]_{w} \ar[rru]^{v}&&&&\circ\ar[rru]^{b_{1}} \ar[rrd]_{b_{3}}\ar[rr]^{b_{2}}&&\circ\\
 &&&&&& \circ
    }}
&
  {\begin{array}{ll}
  \\ [3.5em]
b_1w=0\\ [0.5em]
b_2(w- uv)=0\\ [0.5em]
b_3(w-\lambda uv)=0
\end{array} }
&
{\xymatrix@-1pc{
  &&\co(\vx_1)\ar[rrd]^{X_{1}}&&&& S_{20} \\
  \co\ar[rrrr]_{X_2^2} \ar[rru]^{X_{1}}&&&&\co(\vc)\ar[rru]^{\pi_{2}} \ar[rrd]_{\pi_{4}}\ar[rr]^{\pi_{3}}&&S_{30}\\
 &&&&&& S_{40}
    }} \\

\begin{array}{l}\\[3.5em]
  B_{24}
 \end{array}
 &&
  {\xymatrix@-1pc{
   &&&&\circ\ar[rrd]^{b_{1}}\\
\circ\ar[rr]^{v}&& \circ \ar[rru]^{a_{1}}\ar[rr]^{a_{2}} \ar[rrd]_{a_{3}}&&\circ\ar[rr]^{b_{2}}&&\circ\\
    &&&&\circ\ar[rru]_{b_{3}}
    }}
&
  {\begin{array}{ll}
  \\ [3.5em]
b_{3}a_{3}=b_{2}a_{2}-b_{1}a_{1}\\ [0.5em]
(b_{2}a_{2}-\lambda b_{1}a_{1})v=0
\end{array} }
&
  {\xymatrix@-1pc{
   &&&&\co(\vx_1)\ar[rrd]^{X_{1}}\\
 S_{41}[-1]\ar[rr]^{\eps_{4}}&& \co \ar[rru]^{X_{1}}\ar[rr]^{X_{2}} \ar[rrd]_{X_{3}}&&\co(\vx_2)\ar[rr]^{X_{2}}&&\co(\vc)\\
    &&&&\co(\vx_3)\ar[rru]_{X_{3}}
    }}\\[-1em]

  \begin{array}{l}\\[3.5em]
  B_{31}
 \end{array}
 &&
  {\xymatrix@-1pc{
   &&\circ\ar[rrd]^{b_{1}}&&&&\circ\\
 \circ\ar[rru]^{a_{1}}\ar[rrd]_{a_{2}}&&&& \circ\ar[rru]^{u_{1}}\ar[rrd]_{u_{2}} \\
    &&\circ\ar[rru]_{b_{2}}&&&&\circ
    }}
&
  {\begin{array}{ll}
  \\ [3.5em]
u_{1}(b_{2}a_{2}-b_{1}a_{1})=0\\ [0.5em]
u_{2}(b_{2}a_{2}-\lambda b_{1}a_{1})=0
\end{array} }
&
  {\xymatrix@-1pc{
   &&\co(\vx_1)\ar[rrd]^{X_{1}}&&&&S_{30}\\
\co\ar[rru]^{X_{1}}\ar[rrd]_{X_{2}}&&&& \co(\vc)\ar[rru]^{\pi_{3}}\ar[rrd]_{\pi_4} \\
    &&\co(\vx_2)\ar[rru]_{X_{2}}&&&&S_{40}
    }}\\[-1em]

\begin{array}{l}\\[3.5em]
  B_{32}
 \end{array}
 &&
  {\xymatrix@-1pc{
    \circ\ar[rrd]^{v_{1}}&&&&\circ\ar[rrd]^{b_{1}}\\
 && \circ\ar[rru]^{a_{1}}\ar[rrd]_{a_{2}} &&&&\circ\\
   \circ\ar[rru]_{v_{2}}&&&&\circ\ar[rru]_{b_{2}}
    }}
&
  {\begin{array}{ll}
  \\ [3.5em]
(b_{2}a_{2}- b_{1}a_{1}))v_{1}=0\\ [0.5em]
(b_{2}a_{2}-\lambda b_{1}a_{1})v_2=0
\end{array} }
&
  {\xymatrix@-1pc{
    S_{31}[-1]\ar[rrd]^{\eps_{3}}&&&&\co(\vx_1)\ar[rrd]^{X_{1}}\\
 && \co\ar[rru]^{X_{1}}\ar[rrd]_{X_{2}} &&&&\co(\vc)\\
   S_{41}[-1]\ar[rru]_{\eps_{4}}&&&&\co(\vx_2)\ar[rru]_{X_{2}}
    }}\\[-1em]

\end{array}$

$\begin{array}{cccccl}

 \begin{array}{l}\\[2.5em]
  A_{33}
 \end{array}
 &&
  {\xymatrix@-1.5pc{
    &&&&\circ\ar[rrd]^{b_{1}}\\
 \circ\ar[rr]^{v}&& \circ\ar[rru]^{a_{1}}\ar[rrd]_{a_{2}} &&&&\circ\ar[rr]^{u}&&\circ\\
  &&&&\circ\ar[rru]_{b_{2}}
    }}
&
  {\begin{array}{ll}
  \\ [2.5em]
 (b_{2}a_{2}-b_{1}a_{1})v=0\\ [0.5em]
u(b_{2}a_{2}-\lambda b_{1}a_{1})=0
\end{array} }
&
{\xymatrix@-1.5pc{
    &&&&\co(\vx_1)\ar[rrd]^{X_{1}}\\
 S_{31}[-1]\ar[rr]^{\eps_{3}}&& \co\ar[rru]^{X_{1}}\ar[rrd]_{X_{2}} &&&&\co(\vc)\ar[rr]^{\pi_4}&&S_{40}\\
  &&&&\co(\vx_2)\ar[rru]_{X_{2}}
    }}\\[-2em]

\begin{array}{l}\\[3.5em]
  B_{41}
 \end{array}
 &&
  {\xymatrix@-1pc{
     \circ\ar[rr]^{a_{1}}\ar[rrdd]_<<<{a_{2}}&&\circ\ar[rr]^{b_{1}}\ar[rrdd]^>>>{b_{3}}&&\circ\\
  \\
\circ\ar[rr]_{a_{4}}\ar[rruu]^<<<{a_{3}}&&\circ\ar[rr]_{b_{4}}\ar[rruu]_>>>{b_{2}}&&\circ
    }}
&
  {\begin{array}{ll}
  \\ [3.5em]
b_{2}a_{2}-b_{1}a_{1}=0\\ [0.5em]
b_{4}a_{4}-b_{3}a_{3}=0\\ [0.5em]
b_1a_3-b_{2}a_{4}=0\\ [0.5em]
b_{4}a_{2}-\lambda b_{3}a_{1}=0
\end{array} }
&
{\xymatrix@-1pc{
    \co\ar[rr]^{X_{1}}\ar[rrdd]_<<<{X_{2}}&&\co(\vx_1)\ar[rr]^{\pi_3X_{1}}\ar[rrdd]^>>>{\pi_{4}X_1}&&S_{30}\\
  \\
\co(\vx_1-\vx_2)\ar[rr]_{X_{1}}\ar[rruu]^<<<{X_{2}}&&\co(\vx_2)\ar[rr]_{\pi_4X_{2}}\ar[rruu]_>>>{\pi_{3}X_2}&&S_{40}
    }}\\[-1em]

\begin{array}{l}\\[3.5em]
  A_{42}
 \end{array}
 &&
  {\xymatrix@-1pc{
    &&\circ\ar[rr]^{w}\ar[rrdd]^>>>>>{u}&&\circ\ar[rrd]^{b_{1}}\\
  \circ\ar[rru]^{a_{1}}\ar[rrd]_{a_{2}}&&&&&&\circ\\
&&\circ\ar[rr]_{y}\ar[rruu]_>>>>>{v}&&\circ\ar[rru]_{b_{2}}
    }}
&
  {\begin{array}{lll}
  \\ [3.5em]
b_2y-b_1v=0\\[0.5em]
ya_2-ua_1=0\\[0.5em]
va_2-wa_1=0\\[0.5em]
b_2u-\lambda b_1w=0\\[0.5em]
b_2ya_2=0
\end{array} }
&
  {\xymatrix@-1pc{
    &&\co\ar[rr]^{X_{1}}\ar[rrdd]^>>>>>{X_{2}}&&\co(\vx_1)\ar[rrd]^{\pi_4X_{1}}\\
  S_{31}[-1]\ar[rru]^{\eps_{3}}\ar[rrd]_{X_{2}\eps'_3}&&&&&&S_{40}\\
&&\co(\vx_1-\vx_2)\ar[rr]_{X_1}\ar[rruu]_>>>>>{X_2}&&\co(\vx_2)\ar[rru]_{\pi_4X_{2}}
    }}
\end{array}$
\end{center}}}

In the above list, each algebra $A_{ij}$ or $B_{ij}$ on the first column is given by quiver with relations in the second and third columns respectively. The last column is a realization of each algebra by a tilting object in the bounded derived category $D^b(\coh\X)$ with line bundles and simple sheaves (up to suspension shift). The algebras $A_{ij}$'s are realized by tilting complexes while $B_{ij}$'s are realized by tilting sheaves.

Now we first explain the representatives of arrows ($X_i$, $\pi_i$, $\eps_i$ for $1\leq i\leq 4$ and $\eps'_3$) in the above realizations.

Recall that there is a natural projection from the category of $\mathbb{L}$-graded
$S$-modules to the category of coherent sheaves: $\mod^{\mathbb{L}}(S)\to\coh\X$, where
\[S=k[X_{1}, X_2, X_3, X_4]/(X_{3}^{2}-X_{2}^{2}+X_1^2,\ X_{4}^{2}-X_{2}^{2}+\lambda X_1^2).\]
Under the projection, we use $X_i$'s to denote the obvious multiplication $L\to L(\vx_i)$ for any line bundle $L$ in $\coh\X$. It follows that $X_{3}^{2}=X_{2}^{2}-X_1^2$ and $X_{4}^{2}=X_{2}^{2}-\lambda X_1^2$.

Note that $\Hom(\co(\vc), S_{i0})\cong k$ for $1\leq i\leq 4$. We fix a nonzero element $\pi_i\in\Hom(\co(\vc), S_{i0})$, which fits to the following exact sequence in $\coh\X$: $$\xymatrix@C=0.5cm{
 0\ar[r] & \co(\vx_i) \ar^-{X_i}[r] & \co(\vc) \ar^-{\pi_i}[r]& S_{i0} \ar[r]& 0.}$$
Meanwhile, by Serre duality we have $\Hom(S_{i1}[-1],\co)\cong D\Hom(\co, S_{i0})$, which then has dimension one.
Similarly, we fix a nonzero element $\eps_i\in\Hom(S_{i1}[-1],\co)$, which fits into the following triangle
$$\xymatrix@C=0.5cm{
  S_{i1}[-1]\ar^-{\eps_i}[r] & \co \ar^-{X_i}[r] & \co(\vx_i) \ar[r]& S_{i1}.}$$ It follows that the compositions $\pi_i X_i=0$ and $X_i\eps_i=0$ for any $i$.

Moreover, by using the functor $\Hom(S_{31}[-1],-)$ to the exact sequence
 $$\xymatrix@C=0.5cm{
  0\ar[r] & \co(-\vx_1) \ar^-{X_1}[r] & \co \ar[r]& S_{10} \ar[r]& 0,}$$ the multiplication $X_1: \co(-\vx_1)\to\co$ yields an isomorphism $$\Hom(S_{31}[-1],\co(-\vx_1))\cong\Hom(S_{31}[-1],\co).$$ Hence, there exists an element $\eps'_3\in\Hom(S_{31}[-1],\co(-\vx_1))$ satisfying  $\eps_3=X_1\eps'_3$ and then $X_3\eps'_3=0$.

In each case one proves that the listed complexes are tilting in $D^b(\coh\X)$ having the corresponding algebras as endomorphism algebras. As a typical example, we consider the realization for the algebra $A_{42}$. Take the tilting complex $$T=S_{31}[-1]\oplus\co\oplus\co(\vx_1-\vx_2)\oplus\co(\vx_1)\oplus\co(\vx_2)\oplus S_{41}.$$  Our aim is to show that the endomorphism algebra $\Sigma:=\End(T)$ is isomorphic to $A_{42}$ as described by quiver with relations in List \ref{list}.

It is easy to see that the quiver of $\Sigma$ has the shape as that of $A_{42}$. We now show that the following assignments (comparing the presentations for $A_{42}$ and $\Sigma$ in List \ref{list})
 $$ a_1\mapsto \eps_3;\  a_2\mapsto X_2\eps'_3;\ w\mapsto X_1;\ y\mapsto X_1; v\mapsto X_2;\ u\mapsto X_2;\ b_1\mapsto \pi_4X_1;\ b_2\mapsto \pi_4X_2;$$
 yield an isomorphism of algebras $\sigma: A_{42}=kQ/I\cong\Sigma$,
 where $I$ coincides with the ideal generated by the relations of $A_{42}$ as described in List \ref{list}. In fact, by $\Hom(\co(\vx_1-\vx_2), \co(\vc))\cong k$ we get the following commutative diagram, i.e. $X_1X_2=X_2X_1$:
$$\xymatrix@-1.5pc{
    &&&&\co(\vx_1)\ar[rrd]^{X_{1}}\\
 && \co(\vx_1-\vx_2)\ar[rru]^{X_{2}}\ar[rrd]_{X_{1}} &&&&\co(\vc).\\
  &&&&\co(\vx_2)\ar[rru]_{X_{2}}
    }$$
    It follows that $(\pi_4X_1)X_2=(\pi_4X_2)X_1$, hence $b_2y-b_1v\in I$. Similarly, by the definition of $\eps'_3$ we have $\eps_3=X_1\eps'_3$, it follows that $X_1(X_2\eps'_3)=X_2(X_1\eps'_3)=X_2\eps_3$, hence $ya_2-ua_1\in I$. Moreover, combining $X_{3}^{2}=X_{2}^{2}-X_1^2$ with $X_3\eps'_3=0$, we get $X_2^2\eps'_3- X_1\eps_3=(X_2^2-X_1^2)\eps'_3=0$, hence $va_2-wa_1\in I$; and combining $X_{4}^{2}=X_{2}^{2}-\lambda X_1^2$ with $\pi_4X_4=0$, we get $\pi_4X_{2}^{2}-\lambda\pi_4X_{1}^{2}=0$, hence $b_2u-\lambda b_1w\in I$. Finally, $\Hom(S_{31}[-1], S_{41})=0$ implies $(\pi_4X_2)X_1(X_2\eps'_3)=0$, which yields $b_2ya_2\in I$. Then $\sigma$ is an isomorphism follows immediately.

 We emphasize again that the list of endomorphism algebras in List \ref{list} has already been given by Meltzer \cite[List 10.4]{[M]}, but with a totally different approach.

\end{appendix}

\end{document}